\documentclass[pagesize,pdftex]{article}

\usepackage{latexsym} 
\usepackage{amsmath}
\usepackage{amsthm}
\usepackage[normalem]{ulem}
\usepackage{amsfonts}
\usepackage{amssymb}
\usepackage{amsxtra}
\usepackage{amscd}
\usepackage{ifthen}  
\usepackage{graphicx}
\usepackage{enumerate}    
\usepackage{mathrsfs}
\usepackage[hyphens]{url}
\usepackage{mathtools}
\usepackage[]{quoting}
\usepackage{enumitem}

\usepackage[hidelinks, colorlinks, allcolors = black, breaklinks = true, pagebackref=true,hypertexnames = false]{hyperref} 

\usepackage[usenames,dvipsnames,svgnames,table]{xcolor}

\usepackage{tikz}

\usepackage[utf8]{inputenc}    
\usepackage[T1]{fontenc}       
\usepackage[autostyle]{csquotes}

\usepackage[margin=1in]{geometry}

\usepackage{lipsum}

\newcommand{\ccc}{c}

\newcommand{\Q}{\mathbb Q}
\newcommand{\nin}{\notin} \newcommand{\DS}{\displaystyle}

\DeclareMathOperator{\kerneu}{ker_N}
\newcommand{\kerneun}[2]{{\rm ker}_N^{#1,#2}}

\makeatletter \newcommand{\VEC}[2][r]{
  \gdef\@VORNE{1} \left(\hskip-\arraycolsep%
    \begin{array}{#1}\vekSp@lten{#2}\end{array}%
    \hskip-\arraycolsep\right)}

\newcommand{\eps}{\varepsilon}

\renewcommand{\Re}{\mathop{\mathrm{Re}}}
\renewcommand{\Im}{\mathop{\mathrm{Im}}}

\newcommand{\Rel}{a}

\newcommand{\IP}[2]{\ensuremath{\left\langle\kern-0.25ex\left\langle\kern-0.25ex\left\langle #1, #2 \right\rangle\kern-0.25ex\right\rangle\kern-0.25ex\right\rangle}}
\newcommand{\NNNORM}[1]{\ensuremath{\left\vert\kern-0.25ex\left\vert\kern-0.25ex\left\vert #1 \right\vert\kern-0.25ex\right\vert\kern-0.25ex\right\vert}}
\newenvironment{pdeq}{ \left\{ \begin{aligned}}{\end{aligned}\right.}

\newcommand{\vp}[1]{\ensuremath{\left(#1\right)}}
   
\newcommand{\calb}{{\mathcal B}}
\newcommand{\calc}{{\mathcal C}}
\newcommand{\cald}{{\mathcal D}}

\newcommand{\calf}{{\mathcal F}}

\newcommand{\calh}{{\mathcal H}}

\newcommand{\calk}{{\mathcal K}}
\newcommand{\call}{{\mathcal L}}
\newcommand{\calm}{{\mathcal M}}

\newcommand{\calp}{{\mathcal P}}
\newcommand{\calq}{{\mathcal Q}}

\newcommand{\R}{\mathbb{R}}
\newcommand{\Z}{\mathbb{Z}}
\newcommand{\C}{\mathbb{C}}

\newcommand{\N}{\mathbb{N}}

\DeclareMathOperator{\supp}{supp}

\DeclareMathOperator{\dist}{dist}

\DeclareMathOperator{\sgn}{sgn}

\newcommand{\restr}[1]{{_{|#1}}}

\newcommand{\set}[1]{\ensuremath{\{#1\}}}

\newcommand{\setc}[2]{\ensuremath{\left\{#1\ \colon \ #2\right\}}}

\newcommand{\seqN}[1]{\ensuremath{\set{#1}_{n=1}^\infty}}

\renewcommand{\d}{\mathrm{d}}

\newcommand{\DeltaB}{\Delta}
\newcommand{\norm}[1]{\lVert#1\rVert}

\newcommand{\labs}[1]{\left\lvert #1 \right\rvert}

\newcommand{\opnorm}[1]{{\lvert\kern-0.25ex\lvert\kern-0.25ex\lvert #1 \rvert\kern-0.25ex\rvert\kern-0.25ex\rvert}}

\newcommand{\WSR}[2]{W^{#1,#2}}

\newcommand{\CR}[1]{C^{#1}}

\newcommand{\LR}[1]{L^{#1}}

\newcommand{\LRloc}[1]{L^{#1}_{\mathrm{loc}}} 
 \newcommand{\CRi}{\CR \infty}

\newcommand{\CRci}{\CR \infty_\mathrm{c}}

\newcommand{\newCCtr}[2][d]{
\newcounter{#2}\setcounter{#2}{0}
\expandafter\xdef\csname kyedtheconst#2\endcsname{#1}
}
\newcommand{\Cc}[2][nolabel]{
\stepcounter{#2}
\expandafter\ensuremath{\csname kyedtheconst#2\endcsname_{\arabic{#2}}}
\ifthenelse{\equal{#1}{nolabel}}
{}
{\expandafter\xdef\csname kyedconst#1\endcsname
{\expandafter\ensuremath{\csname kyedtheconst#2\endcsname_{\arabic{#2}}}}}
}
\newcommand{\Ccn}[2][nolabel]{
\expandafter\ensuremath{\csname kyedtheconst#2\endcsname}
\ifthenelse{\equal{#1}{nolabel}}
{}
{\expandafter\xdef\csname kyedconst#1\endcsname
{\expandafter\ensuremath{\csname kyedtheconst#2\endcsname}}}
}

\newcommand{\Cclast}[1]{
\expandafter\ensuremath{\csname kyedtheconst#1\endcsname_{\arabic{#1}}}
}

\newcommand{\Ccllast}[1]{
\addtocounter{#1}{-1}
\expandafter\ensuremath{\csname kyedtheconst#1\endcsname_{\arabic{#1}}}
\addtocounter{#1}{1}
}
\newcommand{\const}[1]{
\expandafter{\ifcsname kyedconst#1\endcsname
  \csname kyedconst#1\endcsname
\else
  \errmessage{Undefined Kyedconstant #1.}%
\fi}
}

\usepackage{array}
\newcolumntype{L}{>{\displaystyle} l <{}}
\newenvironment{TC} {\left \{\begin{array}{LL}} {\end{array} \right.}

\newcommand{\alp}{\alpha}
\newcommand{\gam}{\gamma}
\newcommand{\zet}{\zeta}
\newcommand{\kap}{\kappa}
\newcommand{\lam}{\lambda}
\renewcommand{\phi}{\varphi}

\newcommand{\sig}{\sigma}

\newcommand{\Gam}{\Gamma}

\newcommand{\Sig}{\Sigma}
\newcommand{\Ome}{\Omega}

\newcommand{\scl}[1]{{\mathfrak{s}}_{#1}}
\newcommand{\sclOme}[1]{\scl{#1}^{\Ome}}
\newcommand{\resconst}{\alp_1}
\newcommand{\rpconst}{\alp_0}

\newcommand{\NTL}[2]{\|#1\|_{L^2({#2})}}

\newcommand{\cciL}[1]{C_{\mathrm{c}}^{\infty}(#1)}

\newcommand{\NNN}[2]{\|#1\|_{#2}}

\newcommand{\skpL}[3]{\left(#1, #2\right)_{#3}}

\newcommand{\HS}[1]{H^{#1}_{\alp}}
\newcommand{\HM}[2]{[#1]_{#2,\alp}}
\newcommand{\HMalp}[3]{[#1]_{#2,#3}}
\newcommand{\HN}[2]{|#1|_{#2,\alp}}
\newcommand{\HNalp}[3]{|#1|_{#2,#3}}

\newcommand{\HSalp}[2]{H^{#1}_{#2}}

\newcommand{\HMS}[3]{\left[#1\right]_{#2,#3}}

\newcommand{\HSpol}{\calp} 

\newcommand{\HSL}[2]{\HS{#1}({#2})}

\newcommand{\HSoo}[1]{\mathring{H}^{#1}_{\alp}(\p'\Ome)}
\newcommand{\HSooalp}[2]{\mathring{H}^{#1}_{#2}(\p'\Ome)}

\newcommand{\VS}[1]{H^{#1}_{\alp}(\Ome)}
\newcommand{\VSalp}[2]{H^{#1}_{#2}(\Ome)}

\newcommand{\VSoo}[1]{\mathring{H}^{#1}_{\alp}(\Ome)}
\newcommand{\VSooalp}[2]{\mathring{H}^{#1}_{#2}(\Ome)}

\newcommand{\VSpol}{\calp^{\Ome}}

\newcommand{\PSOme}{\calp_{\Ome}}

\newcommand{\PHSoo}[1]{T\HSoo{k}}

\newcommand{\ZS}[1]{H^{#1}_{\alp}(\Ome)}

\newcommand{\ZM}[2]{[\hspace{-0.3ex}[#1]\hspace{-0.3ex}]_{#2,\alp}}
\newcommand{\ZMalp}[3]{[\hspace{-0.3ex}[#1]\hspace{-0.3ex}]_{#2,#3}}
\newcommand{\ZMM}[3]{[\hspace{-0.3ex}[#1]\hspace{-0.3ex}]_{#2,#3}}

\newcommand{\ZN}[2]{\|#1\|_{#2,\alp}}

\newcommand{\ZNA}[3]{\|#1\|_{#2,#3}}

\newcommand{\ZMS}[3]{[\hspace{-0.3ex}[#1]\hspace{-0.3ex}]_{#2,#3}}
\newcommand{\ZNS}[3]{\left\|#1\right\|_{#2,#3}}

\newcommand{\ZSoo}[1]{{\mathring H}^{#1}_{\alp}(\Ome)}
\newcommand{\ZSooalp}[2]{{\mathring H}^{#1}_{#2}(\Ome)}

\newcommand{\Zskp}[3]{\left(\kern-0.25ex\left(#1, #2\right)\kern-0.25ex\right)_{#3}}

\definecolor{darkcyan}{rgb}{0.0, 0.45, 0.95} 

\definecolor{lightgray}{rgb}{0.8, 0.8, 0.8} 

\definecolor{myviolet}{rgb}{1, 0.5, 0.5} 
\newcommand{\myviolet}{\color{myviolet}}

\usepackage[normalem]{ulem}
\newcommand{\NR}[1]{\small sd}

\newcounter{margcount} 
\renewcommand{\t}{\tilde}

\def\XXint#1#2#3{{\setbox0=\hbox{$#1{#2#3}{\int}$}
\vcenter{\hbox{$#2#3$}}\kern-.5\wd0}}

\theoremstyle{plain}
\newtheorem{theorem}{Theorem}[section]

\newtheorem{lemma}[theorem]{Lemma}

\newtheorem{definition}[theorem]{Definition}
\newtheorem{proposition}[theorem]{Proposition}
\newtheorem{corollary}[theorem]{Corollary}
\newtheorem{remark}[theorem]{Remark}

\newcommand{\dd}{\,\mathrm{d}}

\newcommand{\MOme}{\Omega}
\newcommand{\poly}{\ensuremath{\mathfrak{p}}}
\newcommand{\Mpot}{\ensuremath{\Phi}}

\newcommand{\vvv}{v}
\newcommand{\vv}{\mathfrak{v}}

\renewcommand{\ggg}{g}
\newcommand{\fff}{f}
\newcommand{\gggg}{\mathfrak{g}}
\newcommand{\ffff}{\mathfrak{f}}

\newcommand{\vf}{v}

\newcommand{\ang}{\theta}

\renewcommand{\hat}{\widehat}

\newcounter{MODUS}
\ifnum\theMODUS=1
\newcommand{\DETAILS}[1]{
  {\myviolet{$\langle\hspace{-0.4ex}\langle$ 
      #1$\rangle\hspace{-0.4ex}\rangle$
    }}
}
\else
\newcommand{\DETAILS}[1]{}
\fi
\newcommand{\NODETAILS}[1]{}

\def\vekSp@lten#1{\xvekSp@lten#1;vekL@stLine;} \def\vekL@stLine{vekL@stLine} \def\xvekSp@lten#1;{\def\temp{#1}%
  \ifx\temp\vekL@stLine \else \ifnum\@VORNE=1\gdef\@VORNE{0} \else\@arraycr\fi%
  #1%
  \expandafter\xvekSp@lten \fi} \makeatother

\newcommand{\LL}{\mathcal L}

\newcommand{\SUS}{\subset}
\newcommand{\bet}{\beta}

\newcommand{\Zwischenrechnung}[1]{}

\newcommand{\pa}{\partial}
\newcommand{\p}{\partial}
\newcommand{\OL}{\overline}
\newcommand{\upref}[2]{\hspace{-0.8ex}\stackrel{\eqref{#1}}{#2}} 

\newcommand{\lupref}[2]{\hspace{0ex} \stackrel{\eqref{#1}}{#2}} 

\DeclareMathOperator*{\spann}{span}

\newcommand{\BS}{\backslash}

\DeclareOldFontCommand{\bf}{\normalfont\bfseries}{\mathbf}
 \hypersetup{
 pdfauthor={Marco Bravin, Manuel V. Gnann, Hans Kn\"upfer, Nader Masmoudi, Floris B. Roodenburg, Jonas Sauer},
 pdftitle={Well-Posedness and Regularity of the Heat Equation with Robin boundary conditions in the Two-Dimensional Wedge},
 breaklinks=true,
 colorlinks=true,
 linkcolor=black,
 citecolor=black,
 urlcolor=black,
 filecolor=blue,
 }

\numberwithin{equation}{section} 

\allowdisplaybreaks

\begin{document}


\title{Well-Posedness and Regularity of the Heat Equation with Robin Boundary Conditions in the Two-Dimensional Wedge}

%

\author{ %
  Marco Bravin\footnote{Departmento de Matem\'atica Aplicada y Ciencias de la Computac\'{\i}on, E.T.S.I.~Industriales y de Telecomunicac\'{\i}on, Universidad de Cantabria, 39005 Santander, Spain.
} \and
  Manuel~V.~Gnann\footnote{Delft Institute of Applied Mathematics, Faculty of
    Electrical Engineering, Mathematics and Computer Sciences, Delft University
    of Technology, Mekelweg 4, 2628 CD Delft, Netherlands} \and %
  Hans Kn\"upfer\footnote{Institute of Mathematics and IWR, Heidelberg
    University, INF 205, 69120 Heidelberg, Germany} \and %
  Nader Masmoudi\footnote{Courant Institute, New York, USA, NY 10012, 251 Mercer
    St.  New York} \and %
  Floris~B.~Roodenburg$^\dagger$ \and %
  Jonas~Sauer\footnote{Corresponding author: \texttt{jonas.sauer@uni-jena.de}. Institute for Mathematics, Faculty of Mathematics and
    Computer Science, University of Jena, Ernst-Abbe-Platz 2, 07737 Jena,
    Germany} }

\maketitle
\abstract{Well-posedness and higher regularity of the heat equation with Robin boundary conditions in an unbounded two-dimensional wedge is established in an $\LR{2}$-setting of monomially weighted spaces.
A mathematical framework is developed which allows to obtain arbitrarily high regularity without a smallness assumption on the opening angle of the wedge.
The challenging aspect is that the resolvent problem exhibits two breakings of the scaling invariance, one in the equation and one in the boundary condition.
}


\section{Introduction}\label{sec:int}

We consider for some fixed $\gam \in (0,\infty)$ the inhomogeneous boundary value
problem
  \begin{align}   \label{bvp}%
    \begin{pdeq}
      \p_t U - \Delta U & =  F \qquad &&\text{in $\R_+\times \Ome$,}  \\
      \gam U + \pa_{\nu} U  &= G &&\text{on $\R_+\times \p' \Ome$,} \\
      U\restr{t=0} & = 0 && \text{on $\Ome$}. 
    \end{pdeq} 
  \end{align}
Here,  $\Ome$ (given in polar coordinates) is the two-dimensional wedge
\begin{align*}
  \Ome = \{ r(\cos \phi, \sin \phi) \ : \ r > 0, \phi \in  (0,\ang) \}
\end{align*}
for some given opening angle $\ang \in (0,2\pi)$, $\p'\Ome$ is the boundary of $\Ome$ without the tip $\{0\}\subset\R^2$ and
$\nu$ is the outer unit normal on $\p' \Ome$.
The functions $F=F(t,x)$ and $G=G(t,x)$ are given data, while the function $U=U(t,x)$ is unknown.
We note that there is an extensive literature on boundary value problems for elliptic operators on non--smooth domains, see e.g.~\cite{Brown-1994, Cos19, JeK95, MiM07, She07} and the references therein for general domains and \cite{CSW24, Grisvard2011, KozlovMazyaRossmann-Book} for wedge domains, where techniques based on the Mellin transform have proven to be successful.
Also parabolic boundary problems in the wedge have been studied extensively, see e.g.~\cite{Degtyarev-2010, Fro91, KLS21, Koz88, KoR20, Naz01, PrS07, Solonnikov-1984, FrS91} for a non-exhaustive list.
However, to our knowledge no particular attention is attributed to higher-order regularity in the case of non-scaling invariant problems as the one considered in \eqref{bvp}.
For a bounded domain, the terms with highest scaling are of leading order, while terms of lower lower scaling can be treated by perturbative methods.
This is not evident in the case of an unbounded domain and for a non-scaling invariant operator.
The application of the Mellin transform, as applied to the scaling invariant case, does not directly solve the problem in the inhomogeneous case.
We develop a framework to treat such problems.
For simplicity of the exposition, we consider as a model problem the heat equation with Robin boundary condition as the simplest model with inhomogeneous boundary conditions in the parabolic setting.

\medskip

The resolvent problem corresponding to \eqref{bvp} is coercive in the unweighted energy norm, see Lemma \ref{lem-coercive}.
However, one difficulty to obtain solutions with higher regularity is that the unweighted energy norm is not suitable for applying standard elliptic regularity theory as the domain is not smooth.
In fact, the Neumann Laplacian exhibits certain resonances, by which we mean non-trivial elements in the kernel of the Neumann Laplacian which possess a scaling in the radial variable, see Proposition \ref{prp-ellneumann}.
In order to avoid scalings of the involved seminorms which match those of the resonances, weighted norms are natural to use, cf. \cite{Kondratev-1967}.
Here, the weights are power weights in the distance to the tip of $\Omega$.
To get both existence of weak solutions as well as higher regularity, we work in intersection spaces where both weighted and unweighted norms are controlled.
This approach necessitates a careful analysis, since the transition from weak solutions to classical solutions in this setting is surprisingly non--trivial.
This is related to the fact that the spaces of test functions associated with the intersection type spaces are naturally sum-type spaces. 
In order to show surjectivity in the test function space, we solve a test function problem which is similar to the original problem but has a reduced complexity in terms of scaling invariance.
This method was used in related settings in previous works \cite{BGKMRS24, GnW25}.
In this paper we further develop and highlight this technique for the model \eqref{bvp}.
In particular, we account for all opening angles which do correspond to a resonance\footnote{For this reason, we include the condition $\tfrac\pi\ang\notin\Q$ in Theorems \ref{thm-ex} and \ref{thm-higher}, which guarantees that for every $q\in \calq$ there are unique $j\in \N_0$ and $\ell \in \Z$ such that $q=j+\tfrac\pi\ang\ell$, where $\calq$ is defined in Definition \ref{def-nonhom-norm}.
    In practice we only work with a bounded subset $\tilde\calq\subset \calq$, and the condition $\tfrac\pi\ang\notin\Q$ could be weakenend by only demanding that for every $q\in \tilde\calq$ there are unique $j\in \N_0$ and $\ell \in \Z$ such that $q=j+\tfrac\pi\ang\ell$.} via the quantity $\dist(\alp+1,\frac\pi\ang\Z)$, where $r^{-\alp}$ is a monomial weight in the radial variable $r$.
As mentioned above, the non-scaling invariance of our boundary condition does not allow to  use directly the method from \cite{KozlovMazyaRossmann-Book}.
Instead we use an iterative approach where we successively obtain higher regularity.
The test function in this scheme is used to obtain classical solutions in our intersection spaces as a starting point for the induction argument.

\medskip

Our first main result provides well-posedness of the problem \eqref{bvp} for right hand sides with base regularity in the framework of weighted, fractional Sobolev norms.
These norms have a monomial weight $r^{-\alp}$ in the radial variable $r$ and an exponential weight $e^{-\bet t}$ in time.
We refer to Section \ref{sec:norms} for the precise definitions of these spaces.
Let us emphasize that the unweighted spaces are not suited for higher regularity due to resonances.
Therefore the condition \eqref{ass-alp-0} is natural: The first condition excludes the appearance of such resonances, while the second condition ensures that tools related to Hardy's inequality are available.
Note that we only consider the case of negative exponents $\alp\in (-1,0)$ for the monomial weight.
This is due to the fact that we first construct a variational solution in unweighted spaces.
The transition to weighted spaces then necessitates local control of the weighted norms by the unweighted ones, which translates to negative weights.
We emphasize that the estimates are uniform in the Robin parameter $\gam\in (0,\infty)$, but that the norms depend on $\gam$ in a natural way dictated by scaling.
Indeed, problem \eqref{bvp} and correspondingly Theorems \ref{thm-ex} and \ref{thm-higher} can be reduced to $\gam=1$ by means of the scaled quantities $\t U:=U\circ S_\gam$, $\t F:=\gam^{-2} F\circ S_\gam$, and $\t G:=\gam^{-1} G\circ S_\gam$, where $S_\gam(t,x):=(t/\gam^2,x/\gam)$.
It is this scaling which underlies all norms, and we again refer to Section \ref{sec:norms} for the precise definitions.
\begin{theorem}[Well-Posedness]\label{thm-ex} %
  Let $\ang \in (0,2\pi)$ be such that $\frac \pi \ang \nin \Q$.
  Let $\alp_1\in(0,\infty)$ and suppose that $\alp \in (-1,0)$ satisfies 
  \begin{align} \label{ass-alp-0} %
    \dist(\ang(\alp+1),\pi \Z) \geq \alp_1 \qquad  \text{and} \qquad
    \ang |\alp| \ \geq \ \alp_1.
  \end{align}
  Let $\gamma\in(0,\infty)$ and $\bet \geq \gam^2$ and let
  \begin{align*}
   F &\in \mathbb{F}:=L_{\bet}^2(\ZS{0}) \cap H_{\bet,0}^{\frac\alp 2}(L^2(\Ome)), \\
   G &\in \mathbb{G}:=L_{\bet}^2(\HS{\frac12}(\p'\Ome))\cap H_{\bet,0}^{\frac 14}(\HS{0}(\p'\Ome)) \cap H_{\bet,0}^{\frac\alp 2+\frac14} (L^2(\p \Ome)).
  \end{align*}
  Then there exists a unique solution $U\in \mathbb{E}:=H_{\bet,0}^1(\ZS{0})\cap L_{\bet}^2(\ZS{2})$ to \eqref{bvp}, and it fulfills
  \begin{align*}
    \NNN{U}{\mathbb{E}} + \gam^\frac12\NNN{U}{H_{\bet,0}^{\frac12}(\HS{0}(\p'\Ome))} \lesssim_{\resconst,\ang} \NNN{F}{\mathbb{F}}+\NNN{G}{\mathbb{G}}.
  \end{align*}
  \end{theorem}
  Our second main result shows that the solution exhibits higher regularity if the data does.
  Roughly speaking, we show that regularity of order $\ell\in\N$ for the data translates into regularity order $\ell+2$ for the solution.
  To avoid resonance effects it is natural to make the assumptions in terms of the scaling $\sclOme{\sigma}:=\sigma-1$.
  More precisely, we assume that $(\alp,\ell) \in (-1,0)\times\N$ satisfies 
\begin{align} \label{ass-alp}  %
  \min \big \{ \dist(\ang \sclOme{j+\alp+2},\pi \Z), \ang|\sclOme{j+\alp+1}|, \ang|\sclOme{j+\alp}| \big \} \ \geq \ \alp_1
\end{align}
 for some  $\alpha_1 > 0$ and all $j\in\N_0$ with $j\le \ell$.
\begin{theorem}[Higher Regularity]\label{thm-higher} %
  Let $\ang \in (0,2\pi)$ be such that $\frac \pi \ang \nin \Q$.  Suppose that there are
  $\rpconst,\resconst\in (0,\infty)$ such that $(\alp,\ell)\in (-1,0)\times \N$
  satisfies \eqref{ass-alp} and $|\sclOme{\ell+\alp+2}|\le \rpconst$.
  Let $\gamma\in(0,\infty)$ and $\bet \geq \gamma^2$.
  Suppose that
  \begin{align*}
   F &\in \mathbb{F}_\ell:=\bigcap_{j=0}^\ell H_{\bet,0}^{\frac j 2}(\VS{\ell-2}), \\
   G &\in \mathbb{G}_{\ell+\frac12}:=\bigcap_{j=0}^\ell H_{\bet,0}^{\frac j 2}(\HSL{\ell-j+\frac 12}{\p' \Ome}) \cap H_{\bet,0}^{\frac12(\ell+\frac12)}(\HS{0}(\p'\Ome))\cap H_{\bet,0}^{\frac12(\ell-\frac12)}(\HS{1}(\p'\Ome)).
  \end{align*}
  Then there exists a unique solution
  $U\in \mathbb{E}_{\ell+2}:=\bigcap_{j=0}^{\ell+2} H_{\bet,0}^{\frac j2}(\VS{\ell+2-j})$ to \eqref{bvp}, and it fulfills
  \begin{align*}
    \NNN{U}{\mathbb{E}_{\ell+2}}\lesssim_{\resconst,\rpconst,\ang} \NNN{F}{\mathbb{F}_\ell}+\NNN{G}{\mathbb{G}_{\ell+\frac12}}.
     \end{align*}
\end{theorem}

\begin{remark}
We are confident that our techniques may be combined with a partial Fourier transform in lateral directions to treat the problem at hand in an (actual, higher-dimensional) wedge of the form $\Omega\times \R^d$. Since the main focus of the paper is to introduce a novel method treating non-scaling invariant equations, we present the problem in a two-dimensional setup in order to reduce challenges which relate to known methods to a minimum.
\end{remark}

  The paper is organized as follows.
  In Section \ref{sec:pre} we collect embedding, trace and interpolation estimates relating to Sobolev norms with power weights.
  Section \ref{sec:var} is devoted to establishing a variational solution to the resolvent equation corresponding to \eqref{bvp}.
  In Section \ref{sec:res_par} we provide higher regularity results for the resolvent equation and prove Theorems \ref{thm-ex} and \ref{thm-higher}.

  \section{Preliminaries}\label{sec:pre}

\subsection{Notation and Definition of Spaces}\label{sec:norms}

  By $\N$ we denote the set of natural numbers starting from $1$, and we write $\N_0:=\N\cup\set{0}$.
  $\Q$ represents the rational numbers, $\R$ the real numbers and $\C$ the complex numbers.
  We assume that all functions are by default complex valued.
  If $\mathsf{H}$ and $\mathsf{K}$ are two Hilbert spaces with scalar products $(\cdot,\cdot)_{\mathsf{H}}$ and $(\cdot,\cdot)_{\mathsf{K}}$, respectively, which are continuously embedded into a common Hausdorff space $V$, then we equip $\mathsf{H}\cap \mathsf{K}$ with the scalar product $(\cdot,\cdot)_{\mathsf{H}} + (\cdot,\cdot)_{\mathsf{K}}$, thus turning $\mathsf{H}\cap \mathsf{K}$ into a Hilbert space.
  For $k\in\N_0$, an open subset $O\subset \R^d$ and $O\subset V\subset \OL O$, we denote by $\CR{k}(V)$  the set of $k$-times continuously differentiable functions on $O$ such that all derivatives of order up to $k$ have a continuous extension to $V$. The space $\CR{k}_c(V)$ denotes the subspace of all $f\in \CR{k}(V)$ with support compact in $V$.
  We write $\CRi(V):=\bigcap_{k\in\N_0} \CR{k}(V)$ and $\CRci(V):=\bigcap_{k\in\N_0} \CR{k}_c(V)$.
  
  \medskip
  
  We decompose the boundary of the wedge $\Ome$ into $\p_c \Ome \cup \p_0 \Ome \cup \p_1 \Ome$, where $\p_c\Ome:=\set{0}\subset\R^2$, and where $\p_0\Ome:=\{ r(1, 0) \ : \ r > 0 \}$ and $\p_1\Ome:=\{ r(\cos \ang, \sin \ang) \ : \ r > 0 \}$ are the lower and upper connected component of $\p'\Ome:=\p\Ome \BS \p_c\Ome$, respectively.
  For $\eps > 0$ we define the sector $\Sig_\eps$ as the set of all $z\in \C\BS\set{0}$ with $|\arg z| < \eps$.
  For $M\subset \R$ we define the vertical strip $S_M:=\setc{\lam\in\C}{\Re\lam\in M}$.
  If $M=\set{\bet}$ for one $\bet\in\R$, we simply write $S_\beta$.
  For a scalar-valued function $u$, we denote by $\nabla u$ its gradient, and we use the short-hand notation $|\nabla u|^2 + |\nabla r\nabla u|^2 := |\partial_r u|^2 + |r^{-1}\pa_\phi u| + |\pa_r r\pa_r u|^2 + |\pa_r\pa_\phi u|^2 + |r^{-1}\pa_\phi^2 u|^2$.
  
  \medskip
  
We use weighted Sobolev spaces with integer number of derivatives in the wedge with a power weight $r^{-\alp}$ in the radial variable and their trace spaces on the boundary.
Since these trace spaces have fractional regularity, we define those spaces in terms of the Mellin transform in the radial variable.
For sufficient control of the solution globally in time we use exponential weights in time.
Since we tackle the parabolic equation in terms of its resolvent equation, we use the Laplace transform in the time variable.

\medskip

Let $\mathsf{H}$ be a Hilbert space and
$f\in \LRloc{1}(\R_+,\mathsf{H})$.
Then the Mellin transform (at $\lam\in\C$) and Laplace transform (at $\mu\in \C$) are defined by
\begin{align*}
  \DS \calm f(\lam) \ := \widehat f(\lam) \ := \ \frac{1}{\sqrt{2\pi}} \int_0^\infty r^{-\lam} f(r) \ \frac{\d r}{r}, &&
  \DS\LL f(\mu)  \ := \ \frac 1{\sqrt{2\pi}} \int_{-\infty}^\infty e^{-\mu t} f(t) \dd t.
\end{align*}
The complex number $\lambda$ will always refer to the variable in Mellin space related to the radial variable in physical spaces, while $\mu$ refers to the variable in Laplace space related to the temporal variable in physical spaces.
We refer to Appendix \ref{app-a} for more details about these transforms and their properties.

The properties of weighted spaces are often dictated by an inherent (dimension-dependent) scaling.
We therefore introduce for $\sigma\in \R$  the notation
 \begin{align} \label{def-sigOme} %
   \sclOme{\sigma} := \sigma-1, \qquad %
    \scl{\sigma}:=\sigma-\frac{1}{2} = \ \sclOme{\sigma+\frac 12}.
  \end{align}
Moreover, we fix the Robin parameter $\gam\in (0,\infty)$ in the boundary condition of \eqref{bvp}.
As outlined at the end of the introduction, Theorems \ref{thm-ex} and \ref{thm-higher} will follow from the result for $\gam=1$  by a scaling argument.
For this reason,  we work with $\gam=1$ in all sections below and hence do not include the dependence of the norms and spaces on $\gam$ in our notation.
\begin{definition}[Regular spaces] \label{def-homspace}%
  Let $k,\ell \in \N_0$, $\bet,s\in \R$. For $v\in\LRloc{1}(\Ome)$ and $\psi \in \LRloc{1}(\p\Ome)$ we
  define
   \begin{enumerate}
   \item %
      $\DS \ZMM{v}{\ell}{\bet}^2 \ %
     := \  \sum_{j = 0}^\ell \int_0^\ang \int_0^\infty \labs{r^{-\sclOme{\ell+\bet}} (r\pa_r)^j\pa_\phi^{\ell-j} v(r,\phi)}^2 \,\frac{\d r}{r} \d\phi$.
     \item \label{def-znorm} $\DS \ZNA{v}{k}{\bet}^2 \ %
     := \ \sum_{\ell = 0}^k \gamma^{2(k-\ell)}\ZMM{v}{\ell}{\bet}^2$.
 \item  %
      $\DS\HMalp{\psi}{s}{\bet}^2 \ := \ \HMalp{\psi(0)}{s}{\bet}^2 + \HMalp{\psi(\ang)}{s}{\bet}^2$, where $\HMalp{\psi(\phi)}{s}{\bet}^2 \ := \ %
      \int_{\Re \lam = \scl{s+\bet}} |\lam|^{2s}\big|\widehat \psi(\lam,\phi)\big|^2   \ \d\Im\lam$.  \ %
 \item  \label{seminorm-x}  $\DS \HNalp{\psi}{s}{\bet}^2 \ %
      := \ \gam^{2s}\HMalp{\psi}{0}{\bet}^2  + \HMalp{\psi}{s}{\bet}^2$. 
  \end{enumerate}
  The corresponding weighted inner products are denoted by $\Zskp{\cdot}{\cdot}{k,\bet}$ and $\skpL{\cdot}{\cdot}{s,\bet}$.
  We define the Hilbert spaces $\VSooalp{k}{\bet}$ and $\HSooalp{s}{\bet}$ as the completion of
  $\cciL{\OL \Ome \BS \{ 0 \}}$, respectively $\cciL{\p'\Omega}$, with respect to the
  corresponding norms in \eqref{def-znorm} and \eqref{seminorm-x}.

 \end{definition}
  We give a corresponding representation of $\ZMalp{v}{\ell}{\bet}$ in Mellin variables in Lemma \ref{def-tracenorm}.
  We will also show in Lemma \ref{def-tracenorm-a} below that the spaces $\HSooalp{s}{\bet}$ are indeed trace spaces.
  
  \medskip
  
  For our results in Theorems \ref{thm-ex} and \ref{thm-higher} we need to avoid singularities which depend on the structure of the elliptic operator and also the opening angle.
  In order to capture the singularity of our solutions near the origin, we need
  to allow for polynomial expansions in terms of the radial variable $r$ at the origin.
  Since the spaces $\VSooalp{k}{\bet}$ are defined by density, any $v\in \VSooalp{k}{\bet}$ can be approximated by smooth and compactly supported functions in each seminorm $\ZMalp{v}{\ell}{\bet}$ with $0\le \ell\le k$ (and correspondingly for the spaces on the boundary).
  The following lemma shows that we have corresponding control for norms of lower derivates but same scaling.
  In particular, it implies that $\zet r^\delta \in \VSooalp{k}{\bet}$ for a cut-off function $\zet \in \cciL{[0,\infty)}$ with $1_{[0,1]} \leq \zet \leq 1_{[0,2]}$ if and only if $\delta\ge \sclOme{k+\bet}$.
  For smaller values of $\delta$ the singularity at the origin is too strong to approximate the monomial by a smooth function supported compactly away from the origin.
  \begin{lemma} \label{lem-ZM-dense} %
  Let $k\in\N_0$ and $\bet\in\R$ with
  $\scl{k+\bet- \frac 12} = \sclOme{k+\bet} \neq 0$.  Then
  \begin{enumerate}
  \item $v \in L_{\rm loc}^1(\p_0 \Ome)$ satisfies
    $\sum_{\ell=0}^{k-1} \HMS{\vvv}{k-\frac 12 - \ell}{\bet + \ell} < \infty$ if
    and only if there is a sequence of functions
    $v_n \in \cciL{\p_0 \Ome \BS \{ 0 \}}$ such that
    $\HMalp{v_n - v}{k-\frac 12}{\bet} \to 0$ as $n \to \infty$.
  \item $v \in L_{\rm loc}^1(\Ome)$ satisfies
    $\sum_{\ell=0}^k \ZMS{v}{k-\ell}{\bet+\ell} < \infty$ if and only if there
    is a sequence of functions
    $v_n \in \cciL{\OL \Ome \BS \{ 0 \}}$ such that $\ZMM{v_n - v}{k}{\bet} \to 0$
    as $n \to \infty$.
  \end{enumerate}
\end{lemma}
\begin{proof}
  See proof of Lemma C.2 in \cite{BGKMRS24}.
\end{proof}
  Throughout the rest of the paper, we fix a cut-off function $\zet \in \cciL{[0,\infty)}$ such that $1_{[0,1]} \leq \zet \leq 1_{[0,2]}$.
  The regular spaces only allow for functions which vanish sufficiently quickly at the origin.
  For our solutions, however, we need to allow for functions which have certain singular behaviours close to the origin.
  Indeed, the Laplace operator with Neumann boundary condition has
    an infinite dimensional kernel, cf.~Appendix \ref{sec:neu}, consisting of such singular functions.
  \begin{definition}[Singular spaces] \label{def-nonhom-norm}%
    Let $\ang \in (0,2\pi)$ be such that $\frac\pi\ang\notin\Q$ and define the set of admissible exponents $\calq:=\setc{j+\tfrac{\pi}{\ang} \ell}{j,\ell\in\N_0}$.
    For $\beta,s\in \R$ and $k\in \N_0$ we define the polynomial spaces
    \begin{enumerate}
    \item $\VSpol_{k,\bet} \ := \ \big\{\poly:\Ome\to\C \ | \ \poly(r,\phi) = \sum_{q\in \calq, q<\sclOme{k+\bet}} a_{q}(\phi) r^{q} \ \text{with} \ a_{q} \in  \WSR{k}{2}((0,\ang)) \big\}$, \label{pol-form}
    \item
      $\HSpol_{s,\bet} \ := \ \big\{\poly:\R\to\C \ | \ \poly(r) =
      \sum_{q \in \calq, q<\scl{s+\bet}} a_{q} r^{q} \
      \text{with} \ a_{q} \in \R \big\}$, \label{pol-form-bdry}
    \end{enumerate}
    and equip them with the norms
    \begin{align*}
    \|\poly\|_{\VSpol_{k,\bet}}^2:=\sum_{q\in \calq, q<\sclOme{k+\bet}} \gam^{2(\sclOme{k+\bet}-q)}\|a_q\|_{\WSR{k}{2}((0,\ang))}^2, \qquad \text{respectively} \quad \|\poly\|_{\HSpol_{s,\bet}}^2:=\sum_{q\in \calq, q<\scl{s+\bet}} \gam^{2(\scl{s+\bet}-q)}|a_q|^2.
    \end{align*}
    Moreover, we define the spaces $\VSalp{k}{\bet}:=\VSooalp{k}{\bet}\oplus \zet\VSpol_{k,\beta}$ and $\HSalp{s}{\bet}:=\HSooalp{s}{\bet}\oplus \zet \HSpol_{s,\bet}$, and equip them with the norms
    \begin{enumerate}[resume]
    \item \label{add-norm} %
      $\DS \ZNA{u+\zet\poly_u}{k}{\bet}^2 \ %
        :=  \ \ZNA{u}{k}{\bet}^2 + \|\poly_u\|_{k,\beta}^2$,  %
      \item \label{add-norm-bdr}
        $\DS \HNalp{\psi+\zet\poly_\psi}{s}{\bet}^2 \ := \ \HNalp{\psi}{s}{\bet}^2 + \|\poly_\psi\|_{s,\bet}^2$.
    \end{enumerate}
  \end{definition}

    The fact that $\HMalp{r^\delta}{0}{\bet} \ = \ \infty$ for all $\bet,\delta\in\R$ shows that polynomials are not in the regular spaces.
    There is another element of the kernel of the Laplace operator with Neumann boundary conditions, namely the logarithm $v(r,\phi) = \ln r$.
    Note that the logarithm is not included in our choice of polynomial expansions.
    This is because our approach is to first construct a variational solution which cannot contain
    a logarithm in its expansion by design, and then subsequently showing higher
    regularity results for this variational solution.
    
    \medskip
    
    Finally, we define parabolic spaces with fractional time derivatives and
    vanishing initial data:
    \begin{definition}[Parabolic norms] %
      Let $\mathsf{H}$ be a Hilbert space and let $\bet, s \in \R$.
      For $F \in \cald_0$ where
      \begin{align*}
        \cald_0 \ := \ \{ \phi \in \cciL{\R,\mathsf{H}} \ : \ \phi(t) = 0 \text{\ for $t \leq 0$ } \}
      \end{align*}
      we define the norm
      \begin{align*}
        \NNN{F}{H_{\bet,0}^{s}(\mathsf{H})} \ := \ \Big( \int_0^\infty e^{-\bet t} \NNN{(|\p_t|_\bet + \gam)^{s} F(t)}{\mathsf{H}}^2 \dd t \Big)^{\frac 12} \qquad\qquad %
        \text{where } \ 
        |\p_t|_\bet^s F(t) \ := \ \call_\bet^{-1} (|\cdot|^s \call F)(t).
      \end{align*}
      The space $H_{\bet,0}^{s}(\mathsf{H})$ is defined as the completion of $\cald_0$ with respect to $\NNN{\cdot}{H_{\bet,0}^{s}(\mathsf{H})}$.
      We write $L_{\bet}^2(\mathsf{H}) :=H_{\bet,0}^{0}(\mathsf{H})$.
  \end{definition}

\subsection{Different Characterizations of Norms}

Even though we restrict ourselves to integer derivatives for the norms monitoring the size of the respective quantities in $\Omega$, our proof method requires a corresponding characterization in terms of Mellin variables.
We emphasize that in this section and in the rest of the paper, we will always assume $\gam=1$ for the Robin parameter.
\begin{lemma}[Mellin representation of bulk norm]\label{def-tracenorm}
   For $\ell\in\N_0$, $\alp\in\R$ and $\vvv \in \CRci(\overline{\Ome}\BS\set{0})$
   we have
 \begin{align}\label{eq_mellin_domain}
  \ZM{\vvv}{\ell}^2 \ = \  \sum_{j+m=\ell} & \int_0^\ang \int_{\Re \lam = \sclOme{\ell+\alpha}} \labs{\lam^j \pa_\phi^{m}  \widehat {\vvv}}^2  \d \Im\lam \, \d\phi.
\end{align}
\end{lemma}
\begin{proof}
  For any $\phi\in(0,\ang)$ and $j,m\in \N_0$, we calculate with Lemma \ref{lem-mellin}~\eqref{mel-2} and
  \eqref{mel-plancherel} for $\bet:=\sclOme{\ell+\alpha}$
  \begin{align*}
   \NTL{\lam^{j} \pa_\phi^{m}  \widehat \vvv(\cdot,\phi)}{S_{\bet}}^2  \
   = \NTL{r^{-\bet}(r\pa_r)^{j} \pa_\phi^{m} \vvv(\cdot,\phi)}{\R_+,\frac{\d r}{r}}^2.   
 \end{align*}
 Integrating $\phi$ over $(0,\ang)$ and summing over $j+m = \ell$, we get the
 asserted identity \eqref{eq_mellin_domain}.
\end{proof}

\begin{lemma}[Real space representation of boundary norms] 
  Let $\ell \in \N_0$ and let $\alp \in \R$. Let
  $c=\prod_{j=1}^{\ell} \min\{|\frac{\scl{j+\alp}}{\scl{\ell+\alp}}|,1\}$ and
  $C=\prod_{j=1}^{\ell} \max\{|\frac{\scl{j+\alp}}{\scl{\ell+\alp}}|,1\}$. Then
  for $\psi \in \CRci(\R_+)$ we have
    \begin{enumerate}
    \item  \label{real-spc-plancherel} 
   $\DS \HM{\psi}{\ell}^2 \  = \ \int_0^\infty \big| r^{- \scl{\ell+\alpha}} (r\pa_{r})^{\ell} \psi(r) \big|^2 \ \frac{\d r}{r}$, %
 \item $\DS c \HM{\psi}{\ell} \leq \norm{r^{-\alp}\pa_r^\ell\psi}_{\LR{2}(\R_+)} \leq C \HM{\psi}{\ell}$. \label{est-hardy-5}
    \end{enumerate}  
\end{lemma}
\begin{proof}
  The identity \eqref{real-spc-plancherel} is Plancherel's identity in Lemma \ref{lem-mellin}\eqref{mel-plancherel} in view of Definition \ref{def-homspace}\eqref{seminorm-x} and Lemma \ref{lem-mellin}\eqref{mel-2}.
  Moreover, by Lemma \ref{lem-mellin}\eqref{mel-1} and \eqref{mel-2} we have
  $\widehat{\pa_r^\ell\psi}(\lam)=(\lam+1)(\lam+2)\cdots(\lam+\ell)\widehat{\psi}(\lam+\ell)$,
  so that
  \begin{align*}
   \norm{r^{-\alp}\pa_r^\ell\psi}_{\LR{2}(\R_+)}^2 &= \int_{\Re\lam=\scl{\ell+\alp}} \Big( \prod_{j=1}^{\ell} \frac{ |\lam-\ell+j|}{|\lam|} \Big)^2 \labs{\lam}^{2\ell} |\widehat{\psi}(\lam)|^2\d\Im\lam.
  \end{align*}
  Consequently, \eqref{est-hardy-5} follows.  
\end{proof}

\subsection{Estimates in Homogeneous Spaces} \label{sus-melrep} %

In this section, we state and prove some basic estimates which are useful when
working with the weighted spaces $\HS{k}$ and $\ZS{k}$.
We first recall Hardy's inequality, see e.g.~\cite{Mas11a}.
The following version can be found in \cite[Lemma 5.1]{Knu15}.
 \begin{lemma}[Hardy's inequality] \label{lem-hardy} Let $\bet \neq 0$ and
   suppose that $r^{\bet+1}\p_r\vvv \in L^2(\R_+, \frac{\dd r}r)$. We have
   \begin{align*} 
     \inf_{c \in \R} \NTL{r^\bet (\vvv-c)}{\R_+,\frac{\dd r}r} \
     &\leq \ \bet^{-1} \NTL{r^{\bet+1} \p_r\vvv}{\R_+,\frac{\dd r}r}.  
   \end{align*}
 \end{lemma}
 We provide several estimates relating to the norms in homogeneous spaces:
 \begin{lemma}[Estimates for boundary norms]\label{lem-hardy-mel}
   Let $s,\alp,\bet\in \R$ be such that $\scl{s+\alp}\ne 0$ and $\scl{s+\alp+\bet}\ne 0$.
   We define $c:= \min \big\{|\frac{\scl{s+\alp+\bet}}{\scl{s+\alp}}|^{\sgn s},1 \big
   \}$ and $C:=\max \big \{|\frac{\scl{s+\alp+\bet}}{\scl{s+\alp}}|^{\sgn s},1
   \big \}$.
   For $\vvv \in \CRci((0,\infty))$ we have
  \begin{subequations}
  \begin{align}
   c^{|s|} \HM{r^{-\bet}\vvv}{s} &\leq \HMS{\vvv}{s}{\alp+\bet} \leq C^{|s|} \HM{r^{-\bet}\vvv}{s}, \label{hardy-3} \\
   c^{|s|} \HM{r^{1-\bet}\pa_r\vvv}{s} &\leq \HMS{\vvv}{s+1}{\alp+\bet-1}\leq C^{|s|} \HM{r^{1-\bet}\pa_r\vvv}{s}, \label{hardy-4} \\
   |\scl{s+\alp+\bet}|^{\bet} \HMS{\vvv}{s}{\alp+\bet} &\leq \ \HMS{\vvv}{s+\bet}{\alp}, \label{hardy-5} \\
   c^{|s|}|\scl{s+\alp+\bet+1}|\HM{r^{-\bet-1}\vvv}{s} &\leq \HMS{\pa_r\vvv}{s}{\alp+\bet}. \label{hardy-6}
  \end{align}
  \end{subequations}
 \end{lemma}
 \begin{proof}
   By an elementary calculation for all $z,w \in \C$ with
   $|\Im z| = |\Im w|$ we have
   \begin{align}\label{prf-hardy-mel}
     \min\left\{1,\labs{\frac{\Re z}{\Re w}} \right\} \ %
     \leq \ \labs{\frac{z}{w}}  \ %
     \leq \ \max\left\{1,\labs{\frac{\Re z}{\Re w}} \right\}.
   \end{align}
 Observe that
   $\scl{s+\alp}+\bet=\scl{s+\alp+\bet}$, cf.\@ \eqref{def-sigOme}, and
   $\widehat{r^{-\bet}\vvv}(\lam)=\widehat{\vvv}(\lam+\bet)$, cf.\@ Lemma \ref{lem-mellin}\eqref{mel-1}, so that
  \begin{align*}
   \HM{r^{-\bet}\vvv}{s}^2= \int_{\Re\lam=\scl{s+\alp+\bet}} \left|\frac{\lam-\bet}{\lam}\right|^{2s} \labs{\lam}^{2s} |\widehat{\vvv}(\lam)|^2\d\Im\lam.
  \end{align*}
  Using the definition of $\HM{\vvv}{s}$ in Definition \ref{def-homspace} together with Lemma \ref{lem-mellin} and \eqref{prf-hardy-mel}, we obtain  \eqref{hardy-3}. Estimate \eqref{hardy-4} follows from \eqref{hardy-3} applied to $r\pa_r\vvv$ if we observe $\HMS{r\pa_r\vvv}{s}{\alp+1}=\HM{\vvv}{s+1}$. Similarly,
  \begin{align*}
  \HMS{\vvv}{s}{\alp+\bet}^2
  &= \int_{\Re\lam=\scl{s+\alp+\bet}} \frac1{|\lam|^{2\bet}}  \labs{\lam}^{2s+2\bet} |\widehat{\vvv}(\lam)|^2\d\Im\lam, \\
  \HM{r^{-\bet}\vvv}{s}^2
  &= \int_{\Re\lam=\scl{s+\alp+\bet-1}} \frac{|\lam-\bet+1|^{2s}}{|\lam|^{2s}|\lam+1|^2} \labs{\lam}^{2s} \labs{\lam+1}^2 |\widehat{\vvv}(\lam+1)|^2\d\Im\lam.
  \end{align*}
  These two identities imply \eqref{hardy-5} and \eqref{hardy-6}, respectively.
\end{proof}
\begin{lemma}[Estimates for wedge norms]\label{lem_mellin_domain} %
  Let $\ell \in \N_0, \alp,\bet\in\R$ and let
  \begin{align*}
      b_{\bet}:= 
  \frac 12 \Big(\sum_{j=0}^\ell  
  \max\Big\{\Big|\tfrac{\sclOme{\ell+\alp}}{\sclOme{\ell+\alp+\bet}}\Big|^{j},1\Big\}\Big)^{-1}, \qquad %
  B_{\bet}:= 2 \sum_{j=0}^\ell
  \max\Big\{\Big|\tfrac{\sclOme{\ell+\alp+\bet}}{\sclOme{\ell+\alp}}\Big|^{j},1\Big\}.
  \end{align*}
  Then for $\vvv \in \CRci(\overline{\Ome}\BS\set{0})$ we have with the notation
  $\nabla_{r,\phi} = (\p_r, \frac 1r \p_\phi)$
\begin{subequations}\label{wedge}
  \begin{align}
    b_{\bet} \ZM{r^{-\bet}\vvv}{\ell} &\leq \ZMS{\vvv}{\ell}{\alp+\bet} \leq B_{\bet} \ZM{r^{-\bet}\vvv}{\ell}, \label{wedge-1} \\
  \begin{split}
    b_{\bet+1} \ZM{r^{-\bet}\nabla_{r,\phi}\vvv}{\ell} &\leq \ZMS{\vvv}{\ell+1}{\alp+\bet} \ %
                                                          \leq B_{\bet+1} \ZM{r^{-\bet}\nabla_{r,\phi}\vvv}{\ell}, \label{wedge-2}
  \end{split} \\
   \ZMS{\vvv}{\ell}{\alp+k} &\leq |\sclOme{\ell+\alp+k}|^{-k} \ZMS{\vvv}{\ell+k}{\alp}, \label{wedge-3} \\
  b_{\bet-1}|\sclOme{\ell+\alp+\bet}|\ZM{r^{-\bet}\vvv}{\ell} &\leq \ZMS{\pa_r\vvv}{\ell}{\alp+\bet-1}. \label{wedge-4}
  \end{align}
  Moreover, if $\pa_\phi^m \vvv|_{\pa_0\Ome}=0$ for all $m\in\set{0,\ldots,\ell}$, then we have
  \begin{align}
   b_{\bet-1}\ZM{r^{-\bet}\vvv}{\ell} &\leq \ang \ZMS{r^{-1}\pa_\phi\vvv}{\ell}{\alp+\bet-1}. \label{wedge-5}
  \end{align}
\end{subequations}
 \end{lemma}
\begin{proof}
  The inequalities \eqref{wedge-1}, \eqref{wedge-3} and \eqref{wedge-4} follow
  from Lemma \ref{lem-hardy-mel} upon noting that
  $\scl{j+m+\alp-\frac12}=\sclOme{\ell+\alp}$ if $j + m = \ell$ and thus by
  \eqref{eq_mellin_domain}
 \begin{align*}
  \ZM{\vvv}{\ell}^2= \sum_{j+m=\ell}  \int_0^\ang \HMS{\pa_\phi^{m} \vvv(\cdot,\phi)}{j}{m+\alp-\frac12}^2 \, \d\phi.
 \end{align*}
 Note that $\max\{\ZMS{r\pa_r\vvv}{\ell}{\alp+1},\ZMS{\pa_\phi\vvv}{\ell}{\alp+1}\}\leq \ZM{\vvv}{\ell+1}\leq  \ZMS{r\pa_r\vvv}{\ell}{\alp+1} + \ZMS{\pa_\phi\vvv}{\ell}{\alp+1}$, so that \eqref{wedge-2} follows from \eqref{wedge-1}.
 For \eqref{wedge-5} we use $\pa_\phi^m \widehat{\vvv}(\lam,\phi')=\int_0^{\phi'} \pa_\phi^{m+1} \widehat{\vvv}(\lam,\phi) \d\phi$ and Jensen's inequality to observe
 \begin{align*}
 |\pa_\phi^m \widehat{\vvv}(\lam,\phi)|^2\leq \ang \int_0^{\ang} \labs{\pa_\phi^{m+1} \widehat{\vvv}(\lam,\phi)}^2 \d\phi.
 \end{align*}
This implies
 \begin{align*}
 \int_0^\ang &\int_{\Re \lam = \sclOme{\ell+\alpha}} \labs{ \lam^{j}  \pa_\phi^{m}  \widehat {\vvv}(\lam+\bet,\phi) }^2  \d\Im\lam \, \d\phi \leq \ang^2 \int_0^\ang \int_{\Re \lam = \sclOme{\ell+\alpha}} \labs{ \lam^{j}  \pa_\phi^{m+1}  \widehat {\vvv}(\lam+\bet,\phi) }^2  \d\Im\lam \, \d\phi \\
 &= \ang^2  \int_0^\ang \int_{\Re \lam = \sclOme{\ell+\alpha+\bet-1}} \left|\frac{\lam+1-\bet}{\lam}\right|^j\labs{ \lam^{j}  \pa_\phi^{m}  \widehat {\left(r^{-1}\pa_\phi \vvv\right)}(\lam,\phi) }^2  \d\Im\lam \, \d\phi.
 \end{align*}
 Summing over $j+m=\ell$ gives \eqref{wedge-5} in view of \eqref{eq_mellin_domain} and \eqref{prf-hardy-mel}.
\end{proof}
\begin{lemma}[Interpolation estimates] \label{lem-interp} %
  Let $\ell \in \N_0$.  For $\bet,\bet_1,\bet_2\in\R$ with $\bet_1<\bet<\bet_2$,
  $\eta \in (\bet,\bet+1)$ and for any
  $\vf \in C^\infty_\mathrm{c}(\OL \Ome \setminus \{0\})$ we have
  \begin{enumerate}
  \item \label{lem-cauchy-schwarz-2}
    $\DS \HMS{\vvv}{\ell}{\bet}\le
    \HMS{\vvv}{\ell}{\bet_1}^{\frac{\bet_2-\bet}{\bet_2-\bet_1}}\HMS{\vvv}{\ell}{\bet_2}^{\frac{\bet-\bet_1}{\bet_2-\bet_1}}$,
  \item \label{it-interp-11} %
    $\DS \HMS{\vf}{\ell}{\eta} \ %
    \leq \
    c\HMS{\vf}{\ell}{\bet}^{1+\bet-\eta} \HMS{\vf}{\ell+1}{\bet}^{\eta-\bet}$, \hfill %
  \item \label{lem-cauchy-schwarz-2-wedge}
    $\DS \ZMS{\vvv}{\ell}{\bet}\le
    \ZMS{\vvv}{\ell}{\bet_1}^{\frac{\bet_2-\bet}{\bet_2-\bet_1}}\ZMS{\vvv}{\ell}{\bet_2}^{\frac{\bet-\bet_1}{\bet_2-\bet_1}}$,
  \item \label{it-interp-12-wedge}
    $\DS \ZMS{\vf}{\ell}{\eta} \ \leq \
    c_\Ome \ZMS{\vf}{\ell}{\bet}^{1+\bet-\eta}
    \ZMS{\vf}{\ell+1}{\bet}^{\eta-\bet}$, \hfill %
  \end{enumerate}
  where $c:=|\scl{\ell+\bet+1}|^{\bet-\eta}$ if $\scl{\ell+\bet+1}\neq 0$ and
  $c:=|\scl{\ell+\eta}|^{2(\bet-\eta)}$ otherwise. Furthermore,
  $c_\Ome:=|\sclOme{\ell+\bet+1}|^{\bet-\eta}$ if
  $\sclOme{\ell+\bet+1}\neq 0$ and $c_\Ome:=|\sclOme{\ell+\eta}|^{2(\bet-\eta)}$
  otherwise.
\end{lemma}
\begin{proof}
  \eqref{lem-cauchy-schwarz-2}: Let $p:=\frac{\bet_2-\bet_1}{\bet_2-\bet}$ and $p':=\frac{\bet_2-\bet_1}{\bet-\bet_1}$.
  Then $\frac{1}{p}+\frac{1}{p'}=1$ and
  $\frac{\bet_1}{p}+\frac{\bet_2}{p'}=\bet$ and
 \begin{align*}
   \HMS{\vvv}{\ell}{\bet}^2&= \ \int_0^\infty  [r^{-\sclOme{\ell+\bet}} (r\pa_r)^\ell v(r)]^2 \,\frac{\d r}{r}  \
                   \stackrel{\mathclap{\eqref{def-sigOme}}}{=} \ \int_0^\infty |r^{-\sclOme{\ell+\bet_1}} (r\pa_r)^\ell  v(r)|^{\frac{2}{p}} |r^{-\sclOme{\ell+\bet_2}} (r\pa_r)^j v(r)|^{\frac{2}{p'}} \,\frac{\d r}{r}.
 \end{align*}
 The claim \eqref{lem-cauchy-schwarz-2} thus follows from H\"older's inequality.

 \medskip
 
 \eqref{it-interp-11}: If $\scl{\ell+\bet+1} \neq 0$, then
 \eqref{it-interp-11} is just a combination of \eqref{lem-cauchy-schwarz-2} and
 \eqref{hardy-5}.  We thus assume $\scl{\ell+\bet+1} = 0$.  Then
 $\scl{\ell+\eta} \neq 0$ for all $\eta \in (\bet,\bet+1)$.  We first show
 that for $\eta=\bet+1-2^{-(k+1)}$ for some $k\in \N_0$ we have
  \begin{align*}
  \DS \HMS{\vf}{\ell}{\eta} \ %
    \leq \ 
    2^{k+1}\HMS{\vf}{\ell}{\bet+1-2^{-k}}^{\frac12}
    \HMS{\vf}{\ell+1}{\bet}^{\frac12}.
  \end{align*}
  Indeed, we write $w := (r \p_r)^\ell v$ and obtain from integrating by parts
\begin{align*}
  \HMS{\vf}{\ell}{\eta}^2 %
  &= \ \int_0^\infty r^{-2 \scl{\ell+\eta}} \labs{w}^2 \frac{\d r}{r} \,%
    = \ - \frac{1}{2 \scl{\ell+\eta}} \int_0^\infty r \partial_r (r^{-2 \scl{\ell+\eta}}) \labs{w}^2 \frac{\d r}{r}  \
  = \ \frac{1}{\scl{\ell+\eta}} \int_0^\infty r^{-2 \scl{\ell+\eta}} w \, (r \partial_r w) \, \frac{\d r}{r},  \ %
\end{align*}
so that the assertion follows by an application of the Cauchy--Schwarz inequality and since $|\scl{\ell+\eta}|=2^{-(k+1)}$ in view of $\scl{\ell+\bet+1}=0$.
Iteratively, we then get the estimate \eqref{it-interp-11} for $\eta = \bet + 1 - 2^{-(k+1)}$ with a bound
\begin{align*}
\prod_{j=0}^k 2^{\frac{j+1}{2^{k-j}}}=2^{2k+2^{-k}}\le 2^{(k+1)(2-2^{-k})} = |\scl{\ell+\eta}|^{2(\bet-\eta)},
\end{align*}
where we have used $2k+2^{-k}\le (k+1)(2-2^{-k})$ for $k\in\N_0$ and $2-2^{-k}=\eta-\bet$.
In view of \eqref{lem-cauchy-schwarz-2}, the assertion in \eqref{it-interp-11}
follows for all $\eta \in (\bet,\bet+1)$.

\medskip

\eqref{lem-cauchy-schwarz-2-wedge}, \eqref{it-interp-12-wedge}: The
proofs follow analogously by an additional integration in $\phi$.
\end{proof}
\begin{lemma}[Trace estimates] \label{lem-trace} %
  Let $\ell\in\N$ and $\alp\in \R$.
  Then for all $v\in \CRci(\overline{\Ome}\BS \set{0})$, we have
   \begin{align}\label{trace-0}
     \sup_{\phi' \in [0,\ang]} \HM{v(\cdot,\phi')}{0}^2 \ &\le \ \big(2+\frac{c_\Ome}{\ang}\big)\ZM{v}{0}\ZM{v}{1},%
   \end{align}
   where $c_\Ome:=|\alp|^{-\frac12}$ if
  $\alp\neq 0$ and $c_\Ome:=2$
  otherwise.
  Moreover, it holds and
   \begin{align}\label{trace-1}
     \sup_{\phi' \in [0,\ang]} \HM{v(\cdot,\phi')}{\ell-\frac 12}^2 \ &\le \ \big(2+(\ang|\sclOme{\ell+\alp}|)^{-1}\big)\ZM{v}{\ell}^2.%
   \end{align}
 \end{lemma}
 \begin{proof}
   For $\lam\in \C$, $\widehat v(\phi) := \widehat v(\lam,\phi)$ and for all $\phi',\phi''\in [0,\ang]$ we have
   \begin{align*}
    \labs{\widehat \vvv(\lam,\phi'')}^2 &= 2\Re \int_{\phi'}^{\phi''} \widehat\vvv(\lam,\phi) \,  \OL{\p_\phi \widehat\vvv(\lam,\phi)} \ \d\phi +  \labs{\widehat\vvv(\lam,\phi')}^2.
   \end{align*}
   We now
   \begin{enumerate}
   \item integrate over $\lam\in S_{\scl{\alp}}$ and use the generalized Plancherel identity in Lemma \ref{lem-mellin}\eqref{mel-plancherel} to the effect of $\int_{\phi'}^{\phi''} \int_{\Re\lam=\scl{\alp}} \widehat\vvv \,  \OL{\p_\phi \widehat\vvv} \, \dd\Im\lam \dd\phi = \int_{\phi'}^{\phi''} \int_{\Re\lam=\scl{\alp}} \widehat\vvv(\lam-\frac12,\phi) \,  \OL{\p_\phi \widehat\vvv(\lam+\frac12,\phi)}\, \dd\Im\lam \dd\phi$, or
   \item multiply by $\labs{\lam}^{2\ell-1}\in S_{\sclOme{\ell+\alp}}$ and then integrate over $\lam\in S_{\sclOme{\ell+\alp}}$, respectively,
   \end{enumerate}
   and obtain by the Cauchy-Schwarz inequality and \eqref{eq_mellin_domain} the estimates
   \begin{align*}
      \sup_{\phi' \in [0,\ang]} \HM{v(\cdot,\phi')}{0}^2 \ %
     &\le \ 2\ZM{v}{0}\ZM{\frac1r\p_\phi v}{0} + \inf_{\phi' \in [0,\ang]} \HM{v(\cdot,\phi')}{0}^2,\\
   \sup_{\phi' \in [0,\ang]} \HM{v(\cdot,\phi')}{\ell-\frac 12}^2 \ %
     &\le \ 2\ZMalp{(r\pa_r)^\ell v}{0}{\ell+\alp}\ZMalp{(r\pa_r)^{\ell-1}\p_\phi v}{0}{\ell+\alp} + \inf_{\phi' \in [0,\ang]} \HM{v(\cdot,\phi')}{\ell-\frac 12}^2.
   \end{align*}
   Let now $\eps>0$ and let $\phi'\in [0,\ang]$ be such that for all $\phi\in [0,\ang]$ it holds
   \begin{enumerate}
   \item $\HM{v(\cdot,\phi')}{0}^2\le \HM{v(\cdot,\phi)}{0}^2+\eps$, or
   \item $\HM{v(\cdot,\phi')}{\ell-\frac 12}^2\le \HM{v(\cdot,\phi)}{\ell-\frac12}^2+\eps$, respectively.
   \end{enumerate}
   In the first case, we obtain
   \begin{align*}
    \HM{v(\cdot,\phi')}{0}^2&\le \frac1{\ang} \int_0^\ang \HM{v(\cdot,\phi)}{0}^2 \dd \phi + \eps  = \frac{1}{\ang} \ZMalp{v}{0}{\alp+\frac12}^2 + \eps.
   \end{align*}
   Since $\eps>0$ was arbitrary, we arrive at
   \begin{align*}
    \inf_{\phi' \in [0,\ang]} \HM{v(\cdot,\phi')}{0}^2 &\le \frac{1}{\ang} \ZMalp{v}{0}{\alp+\frac12}^2 \le \frac{c_\Ome}{\ang} \ZM{v}{0}\ZMalp{v}{1}{\alp},.
   \end{align*}
   where we have used Lemma \ref{lem-interp}\eqref{it-interp-12-wedge} in the last step.
   Similarly, in the second case we arrive at
   \begin{align*}
    \inf_{\phi' \in [0,\ang]} \HM{v(\cdot,\phi')}{\ell-\frac 12}^2 
    & \le \frac1{\ang} \ZMalp{(r\pa_r)^\ell v}{0}{\ell+\alp}\ZMalp{(r\pa_r)^{\ell-1}v}{0}{\ell+\alp}\le \frac1{\ang} \ZMalp{v}{\ell}{\alp}\ZMalp{v}{\ell-1}{\alp+1}.
   \end{align*}
   In both cases the combination of the estimate for the supremum and the infimum (and \eqref{wedge-3} of Lemma \ref{lem_mellin_domain} in the second case) yields the result.
 \end{proof}
  The boundary norm in Definition \ref{def-homspace}\eqref{seminorm-x} can be formulated as a trace norm as the next lemma shows.
 We note that the trace estimate in our setting holds in all non--zero integer scalings.
\begin{lemma}[Boundary norms as trace norms]\label{def-tracenorm-a}
  For $\Gam \in \{ \p_0 \Ome, \p_1 \Ome \}$ and $\psi \in \cciL{\Gam}$ let $E_\psi$ be the space of functions $v \in \CRi(\OL\Ome \BS \{0\})$ with $v_{|\Gam} = \psi$.
  Let $\ell\in \N$ and $\alp\in\R$ be such that $\sclOme{\ell+\alp} \neq 0$ and $\ang\labs{\sclOme{\ell+\alp}}\leq \alp_0$.
  Then for all $\psi \in \cciL{\Gam}$ we have
   \begin{align*}
     c \HM{\psi}{\ell-\frac 12} \ %
     \leq  \ \inf_{v \in E_\psi} \ZM{v}{\ell} \ %
     \leq \ C  \HM{\psi}{\ell-\frac 12},
   \end{align*}
   where $c:=\big(2+ (\ang |\sclOme{\ell+\alp}|)^{-1} \big)^{-\frac12}$ and $C := (\ell+1)\max\big\{\rpconst \cosh^2 \rpconst, \frac{\rpconst+\sinh
       \rpconst \cosh \rpconst}{\sinh^2{\rpconst}}\big\}$.
\end{lemma}
\begin{proof}
  Without loss of generality we assume $\Gam = \p_1 \Ome$.
  The lower bound follows directly from \eqref{trace-1}.
  For the upper bound, we note that for $\lam\in \C$ with $\Re\lam=\sclOme{\ell+\alp}$ we have either (a) $|\sin(\lam\ang)|^2 \geq \frac12$ or (b) $|\cos(\lam\ang)|^2 \geq \frac12$.
  Depending on these cases, we choose in Mellin variables either (a) $\hat v$ $=$
  $\sin(\lam\phi) \sin^{-1}(\lam\ang) \hat \psi(\lam)$ or (b) $\hat v$ $=$ $\cos(\lam\phi) \cos^{-1}(\lam\ang) \hat \psi(\lam)$.
  Both definitions yield harmonic extensions $v \in E_\psi$ of $\psi$.
  Hence, there are $f,g\in \set{\cos,\sin}$ such that
     \begin{align*}
       \ZM{v}{\ell}&=\sum_{j+m=\ell}\int_{\Re \lam = \sclOme{\ell+\alp}}\int_0^\ang |\lam^j\p_\phi^m \hat v|^2 \dd\phi\dd\Im\lam \\
       &\leq \ (\ell+1)\int_{\Re \lam = \sclOme{\ell+\alp}} \int_0^\ang \frac{|f(\lam\phi)|^2}{|g(\lam\ang)|^2} \ \dd\phi \ |\lam^\ell \hat \psi|^2 \dd\Im\lam  %
       \lupref{int-coscos}\leq \ C\int_{\Re \lam = \sclOme{\ell+\alp}} |\lam^{\ell-\frac 12}\hat \psi|^2 \dd\Im\lam = \HM{\psi}{\ell-\frac 12}.
     \end{align*}
     This yields the assertion.
   \end{proof}

\section{Variational Solution}\label{sec:var}

In this section we will establish for sufficiently smooth data a variational solution to the resolvent equation
\begin{align} \label{resolvent} %
  \begin{pdeq}
    \mu u - \Delta u \ &= \ f &&\text{in $\Ome$,} \\
    u + \pa_{\nu} u \ &= g &&\text{on $\p \Ome$}, \\
  \end{pdeq}
\end{align}
where $\mu\in \C$ is the complex variable related to a Laplace transform of \eqref{bvp} in time.
The idea is to use a Lax-Milgram argument in an (unweighted) space $H$ (see \eqref{def-space-H}) of sufficient regularity which ensures that \eqref{resolvent} is fulfilled not only in a weak sense, but pointwise almost everywhere.
In order to find a suitable sesquilinear form, we test \eqref{resolvent} with a certain linear combination of derivatives of $v\in H$ which ensures the right amount of smoothness of the solution, see Definition \ref{lem-bilinear}.
However, in order to use the fundamental lemma of calculus of variations to identify the Lax-Milgram solution with a distributional solution to \eqref{resolvent}, we show that the class of these linear combinations of derivatives of $v$ contains $C_c^\infty(\Ome)$ as $v$ runs through $H$, and a similar argument is given for functions on the boundary.
The outline of this section is therefore as follows:
In Section \ref{sec:testfct} we show that the class of test functions is rich enough in the above sense.
In Section \ref{sec:unweighted_var_sol} we use this richness of the test functions to obtain via a Lax-Milgram scheme a variational solution $u\in H$ which at the same time is a distributional solution.
Finally, in Section \ref{sec:we} we update the unweighted information on $u$ to a weighted estimate.

\subsection{Test Function Problem}\label{sec:testfct}

In this section, we will provide certain surjectivity results in the space of test functions.
In Section \ref{sec:unweighted_var_sol} we will define a sesquilinear form in terms of the function $v$, which itself is defined by a smooth and compactly supported function $w$ via the test function problem
\begin{align} \label{tf-resolvent}  
  \begin{pdeq}
  A v  &=  w  \qquad\qquad &&\text{in $\Ome$,} \\
  v  &=  0\qquad &&\text{on $\p \Ome$,}
  \end{pdeq}
\end{align}
where $A:=c_2 r^2\kap + c_0 - c_1 (r \pa_r)^2 - \pa_\phi^2$ and $\kap,c_0, c_1, c_2\in (0,\infty)$ are suitable constants.
The advantage of a sesquilinear form in terms of such a test function is that a Lax-Milgram argument immediately yields a solution with sufficient regularity such that all terms in \eqref{resolvent} are defined pointwise almost everywhere.
In this section we argue that the image of the operator $A$ (which is acting on the dual side of the problem) is large enough to ensure uniqueness for the primal objects.

\medskip

Note that problem \eqref{tf-resolvent} still has a (single) non-scaling invariance which does not allow for a pure Mellin approach.
We therefore want to employ a Lax-Milgram type argument in the Hilbert space $\calh$, defined as the closure of $\CRci(\Ome)$ with respect to the norm
\begin{align*}
\lVert v \rVert_{\calh}^2 &:= \kap \int_\Omega (|v|^2 + |r \nabla v|^2) \, \d x + \int_\Omega (|r^{-1} v|^2 + |\nabla v|^2 + |\nabla r\nabla v|^2) \, \d x.
\end{align*}

\medskip

\begin{proposition}[Test function problem] \label{prp-testfun}
Let $\kap, c_0, c_1, c_2\in(0,\infty)$.
For $w \in \cciL{\Ome}$ there is a unique solution $v \in \calh$ of \eqref{tf-resolvent}, and it holds
  \begin{align} \label{dirichlet-ha} 
    \int_\Ome \big(\kap|v|^2 + |\kap rv|^2 + \kap|r\nabla v|^2 + |r^{-1}v|^2 + |\nabla v|^2 +
    |\nabla r\nabla v|^2 \big) \dd x \ 
    \lesssim \ \int_\Ome |r^{-1} w|^2 \dd x.\ 
  \end{align}
\end{proposition}
\begin{proof}
We introduce another Hilbert space $\t\calh$, defined as the closure of $\CRci(\Ome)$ with respect to the norm
\begin{align*}
\lVert v \rVert_{\t\calh}^2 &:= \lVert v \rVert_{\calh}^2 + \kap\int_\Omega |(r\nabla)^2 v|^2 \, \d x.
\end{align*}
For a parameter $\delta > 0$ (specified below) we define $\t\calb \colon \calh\times\t\calh \to \C$ by
\[
\t\calb(v,\psi) := \int_\Omega r^{-2}A v \, \big(1 -\delta (r\pa_r)^2-\pa_\phi^2\big)\OL\psi\, \d x .
\]
Observe that for $v, \psi \in \CRci(\Ome)$ we obtain through integration by parts that
\begin{align*}
\int_\Omega r^{-2}A v \, \OL\psi \, \d x
&= c_2 \kap \int_\Omega v \, \OL\psi \, \d x + c_0 \int_\Omega r^{-2} v \, \OL\psi \, \d x + c_1 \int_\Omega (\partial_r v) \, (\partial_r \OL\psi) \, \d x \\
&\phantom{=} + \int_\Omega (r^{-1} \partial_\varphi v) \, (r^{-1} \partial_\varphi \OL\psi) \, \d x, \\
- \int_\Omega r^{-2}A v \, ((r \partial_r)^2 \OL\psi) \, \d x
&= c_2 \kap \int_\Omega (r \partial_r v) \, (r \partial_r \OL\psi) \, \d x + 2 c_2 \kap \int_\Omega v \, (r \partial_r \OL\psi) \, \d x + c_0 \int_\Omega (\partial_r v) \, (\partial_r \OL\psi) \, \d x \\
&\phantom{=} + c_1 \int_\Omega (\partial_r r \partial_r v) \, (\partial_r r \partial_r \OL\psi) \, \d x + \int_\Omega (\partial_r \partial_\varphi v) \, (\partial_r \partial_\varphi \OL\psi) \, \d x, \\
- \int_\Omega r^{-2}A v \, (\partial_\varphi^2 \OL\psi) \, \d x
&= c_2 \kap \int_\Omega (\partial_\varphi v) \, (\partial_\varphi \OL\psi) \, \d x + c_0 \int_\Omega (r^{-1} \partial_\varphi v) \, (r^{-1} \partial_\varphi \OL\psi) \, \d x \\
&\phantom{=} + c_1 \int_\Omega (\partial_r \partial_\varphi v) \, (\partial_r \partial_\varphi \OL\psi) \, \d x + \int_\Omega (r^{-1} \partial_\varphi^2 v) \, (r^{-1} \partial_\varphi^2 \OL\psi) \, \d x.
\end{align*}
It follows that there is $C>0$ such that $|\t\calb(v,\psi)| \le C \lVert v \rVert_{\calh} \lVert \psi \rVert_{\calh}$ for all $v\in \calh$ and $\psi \in \t\calh$.
Consequently, $\t\calb$ has a unique extension to a bounded sesquilinear form $\calb:\calh\times\calh\to\C$.
Moreover, for $v \in \calh$ it holds
\begin{align*}
\Re \calb(v,v)
&\ge \frac{\kap \, c_2 }{2} \int_\Omega |v|^2 \, \d x + \kap c_2 \delta (1 - 2 \delta) \int_\Omega |r \partial_r v|^2 \, \d x + \kap \, c_2 \int_\Omega |\partial_\varphi v|^2 \, \d x \\
&\phantom{=} + c_0 \int_\Omega (|r^{-1} v|^2 + \delta |\partial_r v|^2 + |r^{-1} \partial_\varphi v|^2) \, \d x
+ c_1 \int_\Omega (|\partial_r v|^2 + \delta |\partial_r r \partial_r v|^2 + |\partial_r \partial_\varphi v|^2) \, \d x \\
&\phantom{=} + \int_\Omega (|r^{-1} \partial_\varphi v|^2 + \delta |\partial_r \partial_\varphi v|^2 + |r^{-1} \partial_\varphi^2 v|^2) \, \d x.
\end{align*}
Hence, choosing $\delta := \frac14$ we obtain a constant $c > 0$ such that
\begin{align*}
\Re \calb(v,v) \ge c \Big(\kap \int_\Omega (|v|^2 + |r \nabla v|^2) \, \d x + \int_\Omega (|r^{-1} v|^2 + |\nabla v|^2 + |\nabla r\nabla v|^2) \, \d x\Big) \ge c \lVert v \rVert_{\calh}^2.
\end{align*}
In conclusion, $\calb:\calh\times\calh\to\C$ is a bounded and coercive sesquilinear form.

\medskip

Define $\calf \colon \calh \to \C$ for $\psi \in \calh$ through
\[
\calf(\psi) := \int_\Omega r^{-2} w \,  (1 - \delta (r \partial_r)^2 - \partial_\varphi^2) \OL\psi \, \d x.
\]
Then we can estimate
\[
|\calf(\psi)| \le \Big(\int_\Omega |r^{-1} w|^2 \, \d x\Big)^{\frac 1 2} \Big(\int_\Omega |r^{-1} (1 - \delta (r \partial_r)^2 - \partial_\varphi^2) \psi|^2 \, \d x\Big)^{\frac 1 2} \le C_w \lVert \psi \rVert_{\calh}
\]
for a $C_w < \infty$, that is, $\calf$ is a bounded anti-linear functional. The Lax-Milgram theorem entails existence of a unique $v \in \calh$ such that
\[
\calb(v,\psi) = \calf(\psi) \quad \text{for all } \psi \in \calh
\]
and
\begin{align}\label{dirichlet-ha-002}
\int_\Ome \big(\kap|v|^2 + \kap|r\nabla v|^2 + |r^{-1}v|^2 + |\nabla v|^2 +
    |\nabla r\nabla v|^2 \big) \dd x \ 
    \lesssim \ \int_\Ome |r^{-1} w|^2 \dd x.\ 
\end{align}
Since $v \in \calh$, due to the definition of $\calh$ and by taking traces, we conclude that the boundary condition in \eqref{tf-resolvent} is satisfied.
For $\psi\in \t\calh$ we obtain $\t\calb(v,\psi)=\calb(v,\psi)=\calf(\psi)$, so that
\[
\int_\Omega r^{-2}(A v - w) \, (1 - \delta (r \partial_r)^2 - \partial_\varphi^2) \OL\psi \, \d x = 0 \quad \text{for all } \psi \in \t\calh.
\]
In order to conclude that the first line in \eqref{tf-resolvent} is satisfied, we thus need to show that for each $\Phi \in C^\infty_\mathrm{c}(\Omega)$ there is $\psi\in\t\calh$ with
\begin{align}\label{problem-psi-test}
\begin{pdeq}
(1 - \delta (r \partial_r)^2 - \partial_\varphi^2) \psi &= \Phi && \text{in } \Omega, \\
\psi &= 0 && \text{on } \partial\Omega.
\end{pdeq}
\end{align}
We use the Mellin transform in $r$ and expand in a sine Fourier series in the angle $\varphi \in (0,\theta)$, that is,
\[
\hat\psi(\lambda,\varphi) = \sum_{k = 1}^\infty \hat \psi_k(\lambda) a_k(\varphi), \quad \text{where} \quad \hat\psi_k(\lambda) := \int_0^\theta \hat{\psi}(\lambda,\varphi) \, a_k(\phi) \, \d\phi \quad \text{and} \quad a_k(\varphi) := \sqrt{\tfrac 2 \theta} \sin\big(\tfrac{k \pi}{\theta} \varphi\big),
\]
satisfying \eqref{problem-psi-test} on taking
\[
\hat\psi_k(\lambda) := \frac{\hat\Phi_k(\lambda)}{1 - \delta \lambda^2 - (\tfrac{k \pi}{\theta})^2}, \quad \text{where} \quad \hat\Phi_k(\lambda) := \int_0^\theta \hat{\Phi}(\lambda,\varphi) \, a_k(\phi) \, \d\phi.
\]
Using the Plancherel identity for the Mellin transform and Parseval's identity for the sine Fourier series, we have
\begin{align*}
\lVert \psi \rVert_{\t\calh}^2 &\sim \int_0^\theta \int_0^\infty (\kap r^2 +1)(|\psi|^2 + |r \partial_r \psi|^2 + |\partial_\varphi \psi|^2 + |(r \partial_r)^2 \psi|^2 + |r \partial_r \partial_\varphi \psi|^2 + |\partial_\varphi^2 \psi|^2) \, \tfrac{\d r}{r} \, \d\varphi \\
&= \kap \int_0^\theta \int_{\Re\lambda = -1} \big((1 + |\lambda|^2 + |\lambda|^4) \, |\hat\psi(\lambda,\varphi)|^2 + (1 + |\lambda|^2) \, |\partial_\varphi \hat\psi(\lambda,\varphi)|^2 + |\partial_\varphi^2 \hat\psi(\lambda,\varphi)|^2\big) \, \d\Im\lambda \, \d\varphi \\
&\phantom{=} + \int_0^\theta \int_{\Re\lambda=0} \big((1 + |\lambda|^2 + |\lambda|^4) \, |\hat\psi(\lambda,\varphi)|^2 + (1 + |\lambda|^2) \, |\partial_\varphi \hat\psi(\lambda,\varphi)|^2 + |\partial_\varphi^2 \hat\psi(\lambda,\varphi)|^2\big) \, \d\Im\lambda \, \d\varphi \\
&= \kap \sum_{k = 1}^\infty \int_{\Re\lambda = -1} \big(1 + |\lambda|^2 + |\lambda|^4 + (1 + |\lambda|^2) \, (\tfrac{k\pi}{\theta})^2 + (\tfrac{k\pi}{\theta})^4\big) \, |\hat\psi_k(\lambda)|^2 \, \d\Im\lambda \\
&\phantom{=} + \sum_{k = 1}^\infty \int_{\Re\lambda=0} \big(1 + |\lambda|^2 + |\lambda|^4 + (1 + |\lambda|^2) \, (\tfrac{k\pi}{\theta})^2 + (\tfrac{k\pi}{\theta})^4\big) |\hat\psi_k(\lambda)|^2 \, \d\Im\lambda \\
&= \kap \sum_{k = 1}^\infty \int_\R \frac{3 + s^2 + s^4 + (2+s^2) (\tfrac{k\pi}{\theta})^2 + (\tfrac{k\pi}{\theta})^4}{(1-\delta - (\tfrac{k \pi}{\theta})^2 + \delta s^2 - 2i\delta s)^2} \, |\hat\Phi_k(-1+is)|^2 \, \d s \\
&\phantom{=} + \sum_{k = 1}^\infty \int_\R \frac{1 + s^2 + s^4 + (1 + s^2) \, (\tfrac{k\pi}{\theta})^2 + (\tfrac{k\pi}{\theta})^4}{(1 - (\tfrac{k \pi}{\theta})^2 + \delta s^2)^2} \, |\hat\Phi_k(is)|^2 \, \d s \\
&\lesssim_{\theta} \kap \sum_{k = 1}^\infty \int_\R |\hat\Phi_k(-1+is)|^2 \, \d s + \sum_{k = 1}^\infty \int_\R |\hat\Phi_k(is)|^2 \, \d s = \kap\int_\Omega |\Phi|^2 \, \d x + \int_\Omega r^{-2} |\Phi|^2 \, \d x < \infty.
\end{align*}
Hence, $\psi \in \t\calh$ and therefore $v$ fulfills \eqref{tf-resolvent}.
In particular
\begin{align*}
|\kap r v|^2 = c_2^{-2}|r w- r Av|^2\lesssim |rw|^2 + |r^{-1}v|^2 + |\nabla v|^2 + |\nabla r\nabla v|^2,
\end{align*}
so that \eqref{dirichlet-ha-002} updates to \eqref{dirichlet-ha}.
\end{proof}

  \begin{lemma}\label{prp-testfun-bdry}
   Let $\kap,c_0,c_1,c_2>0$.
   Then for each $\eta\in \CRci(\R_+)$ there is $\rho\in \CRi(\R_+)$ such that $\big(c_2 \kappa + r^{-2}(c_0 - c_1(r\pa_r)^2)\big)\rho=\eta$ and for each $\ell\in\Z$ with $2c_1\ell^2<c_0$ and each $j\in\N_0$ it holds
   \begin{align*}
   \kap\int_0^\infty |r^{\ell+1}(r\pa_r)^j\rho|^2 \, \frac{\d r}{r} + \int_0^\infty |r^\ell(r\pa_r)^j\rho|^2 \, \frac{\d r}{r} <\infty.
   \end{align*}
  \end{lemma}
\begin{proof}
 Introduce the Hilbert space $\calk$ as the closure of $\CRci(\R_+)$ with respect to the norm $\lVert\cdot\rVert_{\calk}$, where
\begin{align*}
  \NNN{\rho}{\calk}^2 \ 
  &= \  \kap\int_0^\infty |r^{\ell+1}\rho|^2  \, \frac{\d r}{r} + \int_0^\infty (|r^\ell \rho|^2 + |r \pa_r r^\ell \rho|^2) \, \frac{\d r}{r}\\
  & \cong_\ell  \kap\int_0^\infty |r^{\ell+1}\rho|^2  \, \frac{\d r}{r} + \int_0^\infty (|r^\ell \rho|^2 + |r^\ell(r \pa_r) \rho|^2) \, \frac{\d r}{r}.
\end{align*}
We define $\calc \colon \calk \times \calk \to \C$ by
\begin{align*}
\calc(\rho,\psi) &:= c_2 \kap \int_0^\infty  r^{2\ell+2} \rho \, \OL\psi \, \frac{\d r}{r} + (c_0 - c_1\ell^2) \int_0^\infty r^{2\ell} \rho \, \OL\psi \, \frac{\d r}{r} \\
&\quad - 2\ell c_1 \int_0^\infty   r^\ell \rho \, (r\partial_r r^\ell \OL\psi) \, \frac{\d r}{r} + c_1 \int_0^\infty  (r\partial_r r^\ell \rho) \, (r\partial_r r^\ell \OL\psi) \, \frac{\d r}{r}.
\end{align*}
Clearly $\calc(\rho,\psi)\lesssim \|\rho\|_\calk \|\psi\|_\calk$, and by Young's inequality we obtain for all $\eps>0$
\begin{align*}
\calc(\rho,\rho) &:= c_2 \kap \int_0^\infty  |r^{\ell+1} \rho|^2 \, \frac{\d r}{r} + (c_0 - (1+\eps^{-2})c_1\ell^2) \int_0^\infty |r^{\ell} \rho|^2 \, \frac{\d r}{r}  + c_1(1-\eps^2) \int_0^\infty  |r\partial_r r^\ell \rho|^2 \, \frac{\d r}{r}.
\end{align*}
Choosing $\eps\in (0,1)$ sufficiently close to $1$ such that $c_0 - (1+\eps^{-2})c_1\ell^2>0$ (recall that $2c_1\ell^2<c_0$ by assumption), we may employ the Lax-Milgram theorem and obtain a unique $\rho\in \calk$ such that $\calc(v,\psi)=\int_0^\infty r^{2\ell+2}\eta \OL\psi \frac{\dd r}{r}$ for all $\psi\in \calk$, in particular for all $\psi\in \CRci(\R_+)$.
Integrating by parts, we thus learn that
\begin{align*}
 \big(c_2 \kappa + r^{-2}(c_0 -\ell^2 - 2\ell c_1 (r\pa_r) - c_1(r\pa_r)^2)\big)r^\ell\rho=r^\ell\eta
\end{align*}
in the sense of distributions, that is $\big(c_2 \kappa + r^{-2}(c_0 - c_1(r\pa_r)^2)\big)\rho=\eta$.
Observe that quantitatively, only the information $\int_0^\infty |r^{\ell+1}\eta|^2 \frac{\dd r}{r}<\infty$ was used.
Hence, since $\rho_j:=(r\pa_r)^j \rho$ solves
\begin{align*}
 \big(c_2 \kappa + r^{-2}(c_0 - c_1(r\pa_r)^2)\big)\rho_j =(r\pa_r)^j\eta + c_1\sum_{m=0}^{j-1} \binom{j}{m} 2^{j-m} (r\pa_r)^{m} \rho
\end{align*}
for any $j\in\N$, the same argument yields iteratively the estimate for the higher derivatives.
\end{proof}

\subsection{Unweighted Variational Solutions with Higher Regularity}\label{sec:unweighted_var_sol}

Fix $\eps\in(0,\pi)$ and $\mu\in \Sigma_{\pi-\eps}$.
For $\kap:=|\mu|>0$ consider the space
\begin{align} \label{def-space-H} %
  H \ = \ \OL{\cciL{\OL \Ome\setminus\set{0}}}^{\NNN{\cdot}{H}},
\end{align}
where the norm $\NNN{\cdot}{H}$ is given by
\begin{align*}
  \NNN{u}{H}^2 \ 
  &= \  \kap\int_\Ome \Big( |u|^2 + \kap |ru|^2+ |r\nabla u|^2 \Big) \dd x + \int_\Ome \Big( |\nabla u|^2  +  |\nabla r\nabla  u|^2 \Big) \dd x + \int_{\p \Ome} \Big(  |u|^2 + \kap |r u|^2 + |r\pa_r u|^2 \dd r\Big).
\end{align*}
Here we write $\int_{\partial\Omega} f \dd r = \int_0^\infty f(r,0) + f(r,\ang) \dd r$.
We note that the space $H$ does not depend on $\kap > 0$.
Note that all terms in the norm have the same scaling if we use parabolic scaling in the sense that $\kap$ scales like the square of the inverse length.
\begin{definition}[Sesquilinear form] \label{lem-bilinear} For $c_0, c_1, c_2 \in \R$, we define
  $B : H \times \cciL{\OL \Ome\setminus\set{0}} \to \C$ by
  \begin{align*}
    B(u,v) \ 
    &:= \ \int_\Ome(\mu u  - \Delta u ) (c_0 - c_1 (r \p_r)^2   + c_2  |\mu| r^2 - \p_\phi^2) \OL v \dd x \\
    & \qquad + \int_{\p \Ome}(\gam u  + \p_\nu u) (c_0 - c_1 (r \p_r)^2 + c_2  |\mu| r^2) \OL v \dd r.
  \end{align*}
\end{definition}
Since $\mu u-\Delta u\in \LRloc{1}(\OL\Ome\setminus\set{0})$ and $\gamma u+\p_\nu u\in \LRloc{1}(\pa\Ome\setminus\set{0})$ for $u\in H$, the sesquilinear form is well--defined.
Using integration by parts we can show that the sesquilinear form has a unique continuous extension which is coercive on $H \times H$.
\begin{lemma}[Continuity and Coercivity]\label{lem-coercive} %
  Let  $c_0, c_1, c_2 \in \R$ and let $B$ be as in Definition \ref{lem-bilinear}.
  \begin{enumerate}
  \item\label{lem-coercive-i} There is a unique continuous extension $B : H \times H \to \C$.
  \item\label{lem-coercive-iii} For $u,v\in H$ with $v\restr{\p\Ome}=0$ it holds $B(u,v)=\int_\Ome(\mu u  - \Delta u ) (c_0 - c_1 (r \p_r)^2   + c_2  |\mu| r^2 - \p_\phi^2) \OL v \dd x$.
  \item\label{lem-coercive-ii} For $c_0 \gg c_1 \gg c_2 \gg 1$, we have $|B(u,u)| \ \gtrsim_\eps \ \NNN{u}{H}^2$ for all $u\in H$.
  \end{enumerate}
\end{lemma}
\begin{proof}
  The proof rests on the identity
  \begin{align}\label{bilid}
  \begin{split}
    B(u,v) \ = \ \t B(u,v) &+ c_1\Big(2\int_\Ome \mu u r\p_r\OL v \dd x+ \int_{\p \Ome} u r \pa_r \OL v \dd r\Big) \\
    &+   2c_2 |\mu|\int_\Ome  r \pa_r u \OL v \dd x- \int_\Ome (\mu u - \p_r^2 u) \pa_\phi^2 \OL  v \dd x
 \end{split}
  \end{align}
  with
  \begin{align*} 
    \t B(u,v) \ &= \ c_0 \Big( \int_\Ome \mu  u \OL v \dd x+  \int_\Ome \nabla u \cdot \nabla \OL v \dd x+ \int_{\p \Ome} u \OL v \dd r\Big) \\
    &\qquad+ \ c_1 \Big( \int_\Ome \mu (r \pa_r  u) (r \pa_r \OL v) \dd x+ \int_\Ome \nabla   r \pa_r u \cdot \nabla r \pa_r  \OL v \dd x+ \int_{\p \Ome} (r \pa_r u) (r \pa_r \OL v) \dd r\Big) \\
    &\qquad+ \ c_2|\mu| \Big(  \int_\Ome \mu r^2 u  \OL v \dd x+  \int_\Ome  (r \nabla u) \cdot (r \nabla   \OL v)  \dd x+  \int_{\p \Ome}  r^2  u  \OL v \dd r\Big) \\
    &\qquad + \  \int_\Ome r^{-2} \pa_\phi^2 u \pa_\phi^2 \OL  v \dd x,
  \end{align*}
  which we will establish for $u,v \in \cciL{\overline\Ome\setminus\set{0}}$ below.

  \medskip
  
  Assuming that \eqref{bilid} holds, assertions \eqref{lem-coercive-i} and \eqref{lem-coercive-iii} follow immediately by density of $\cciL{\overline\Ome\setminus\set{0}}$ in $H$.
  Furthermore, we note that $\t B(u,u)$ has the form $\mu a^2 + b^2$ with $a, b \in \R$ (where $a, b$ depend on $|\mu|$).
  We thus can use Lemma \ref{lem-complane} to estimate $|\mu a^2 + b^2|$ $\gtrsim_\eps$ $|\mu| a^2 + b^2$ and get by an application of Young's inequality
  \begin{align*} 
    |B(u,u)| \ 
    &\gtrsim_\eps \ c_0 \Big( \int_\Ome |\mu|  |u|^2 + |\nabla u|^2 \dd x+ \int_{\p \Ome} |u|^2 \dd r\Big) \\
    &\qquad +  c_1 \Big( \int_\Ome \frac 12 |\mu| |r \pa_r  u|^2 - 2 |\mu| |u|^2 + |\nabla   r \pa_r u|^2 \dd x + \frac12 \int_{\p \Ome} |r \pa_r u|^2 - \frac12 |u|^2 \dd r\Big) \\
    &\qquad +  c_2 \Big(  \int_\Ome |\mu|^2 r^2 |u|^2 +  \frac 12  |\mu| |r \nabla u|^2 -  2  |\mu|   |u|^2 \dd x  + \int_{\p \Ome} |\mu| r^2  |u|^2 \dd r\Big) \\
    &\qquad + \frac 12 \int_\Ome \frac 1{r^2} |\p_\phi^2 u|^2 - r^2 |\mu|^2 |u|^2  -  |r \p_r^2 u|^2 \dd x. 
  \end{align*}
  For $c_0 \gg c_1 \gg c_2 \gg 1$ the negative terms on each line can then be absorbed by positive terms on the lines above.
  The positive terms yield the desired lower bound in \eqref{lem-coercive-ii}. 

  \medskip

  It remains to show \eqref{bilid}.
  We define $f := \mu u - \Delta u$ and $g := u + \pa_\nu u$.
  Testing $f$ with $v$ we get
  \begin{align*}
  \int_\Ome f \OL v \dd x 
    &=  \int_\Ome (\mu u - \Delta u) \OL v \dd x 
      =  \int_\Ome \mu  u \OL v \dd x + \int_\Ome \nabla u \cdot \nabla \OL v \dd x - \int_{\p \Ome} \p_\nu u  \OL v \dd r
  \end{align*}
  By the definition of $g$ this yields
  \begin{align} \label{test-v} %
    \mu \int_\Ome u \OL v \dd x + \int_\Ome \nabla u \cdot \nabla \OL v \dd x+ \int_{\p \Ome} u  \OL v \dd r\ %
    = \ \int_\Ome f \OL v \dd x+ \int_{\p \Ome} g \OL v \dd r. %
  \end{align}
  Before we continue, we first note that
  \begin{align*}
    \int_\Ome  u (r\pa_r)  \OL v  \dd x \ \
    &
    = \   - \int_\Ome (r\pa_r + 2) u  \OL v \dd x, \\ 
    \int_0^\infty  (r\pa_r u)   \OL v  \dd r \ \
    &
    = \ - \int_0^\infty  u  (r \pa_r + 1)  \OL v  \dd r,    \\
    (r \p_r + 2) \Delta u \ 
    &
      = \ \Delta (r \p_r) u, \\
    \p_\nu (r\p_r )u &= (r\p_r + 1) \p_\nu u.
  \end{align*}
  We next test with $- (r \p_r)^2 v$. Using the above identities we get
  \begin{align*}
    \hspace{2ex} & \hspace{-2ex} 
                   - \int_\Ome f (r \p_r)^2 \OL v \dd x  
    =   - \int_\Ome (\mu u - \Delta u) (r \p_r)^2 \OL v \dd x 
    =  \int_\Ome (r \p_r + 2) (\mu u - \Delta u) (r \p_r) \OL v \dd x\\
    &= \  \int_\Ome \mu (r \p_r + 2) u r \p_r \OL v \dd x - \int_\Ome (\Delta r \p_r u)  (r \p_r \OL v) \dd x\\
    &= \ \int_\Ome \mu  (r \p_r + 2) u r \p_r \OL v \dd x+ \int_\Ome \nabla  r \p_r u  \cdot \nabla r \p_r \OL v \dd x
      - \int_{\p \Ome} (\p_\nu  r \p_r u)  (r \p_r \OL v) \dd r\\
    &= \ \int_\Ome \mu  (r \p_r + 2) u r \p_r \OL v \dd x+ \int_\Ome \nabla  r \p_r u  \cdot \nabla r \p_r \OL v \dd x
      + \int_{\p \Ome} (r \p_r + 1) u  r \p_r \OL v \dd r
      - \int_{\p \Ome} (r \p_r + 1) g   r \p_r \OL v \dd r.
  \end{align*}
  Using $- \int_{\p \Ome} (r \p_r + 1) g   r \p_r \OL v \dd r=\int_{\p \Ome} g   (r \p_r)^2 \OL v \dd r$ and rearranging the terms, we thus learn
  \begin{align} \label{test-rdrv}
  \begin{split}
    \mu \int_\Ome (r \p_r + 2) u r \p_r \OL v \dd x&+ \int_\Ome \nabla  (r \p_r) u  \cdot \nabla (r \p_r) \OL v \dd x
      + \int_{\p \Ome} (r \p_r +1)u  r \p_r \OL v \dd r\\ %
      &= \ -\int_\Ome f (r \p_r)^2 \OL v \dd x- \int_{\p \Ome} g (r \p_r)^2 \OL v \dd r. %
  \end{split}
  \end{align}
  We also test with $r^2 v$. We calculate
  \begin{align}\label{test-rrmuv}
  \begin{split}
    \int_\Ome f r^2 \OL v \dd x 
    &=  \int_\Ome (\mu u - \Delta u) r^2 \OL v \dd x 
    = \ \mu \int_\Ome r^2 u \OL v  \dd x+ \int_{\Ome} \nabla  u  \cdot \nabla (r^2 \OL v) \dd x - \int_{\p \Ome} r^2 \p_\nu u \OL v \dd r\\ 
    &= \ \mu \int_\Ome r^2 u \OL v \dd x + \int_{\Ome} r^2 \nabla  u  \cdot \nabla \OL v \dd x+ 2 \int_{\Ome} r \p_r  u  \cdot \OL v \dd x+ \int_{\p \Ome} r^2 u \OL v \dd r- \int_{\p \Ome} r^2 g \OL v \dd r. 
  \end{split}
  \end{align}
  Finally,  we test the equation with $\p_\phi^2 v$ to get
  \begin{align} \label{test-phiphi} 
    - \int_\Ome f \p_\phi^2 \OL v \dd x =  - \int_\Ome (\mu u - \Delta u) \p_\phi^2 \OL v \dd x= \int_\Ome r^{-2} \p_\phi^2 u\p_\phi^2 \OL v \dd x- \int_\Ome (\mu u - \pa_r^2 u) \p_\phi^2 \OL v \dd x. 
  \end{align}
  If we add the identities $c_0$\eqref{test-v} + $c_1$\eqref{test-rdrv} + $c_2 |\mu|$\eqref{test-rrmuv} + \eqref{test-phiphi} we
  obtain the asserted identity.
\end{proof}
In order to apply the Lax--Milgram theorem, it is vital that the process of adding different derivatives of test functions in the sesqulinear form $B$ still yields a class of functions which engulfs $\CRci(\OL\Ome\BS\set{0})$ and is thus dense in $H$.
This was the purpose of Section \ref{sec:testfct}.
We make this precise in the following lemma.
\begin{lemma}[Variational solution]\label{lem-ex-sol} 
  Let $\eps\in(0,\pi)$ and $\mu\in \Sigma_{\pi-\eps}$.
  Suppose that $f\in \CRci(\OL\Ome\BS\set{0})$ and $g\in \CRci(\pa'\Ome)$.
  Then there exists a unique classical solution $u \in H$ to \eqref{resolvent}, and it holds
  \begin{align*}
  \|u\|_H + \ZMalp{\Delta u}{0}{0} + \HMalp{\p_\nu u}{0}{0}\lesssim \ZMalp{f}{0}{0}+\ZMalp{f}{0}{-1} + \HMalp{g}{0}{0}+\HMalp{g}{1}{-1}.
  \end{align*}
\end{lemma}
\begin{proof}
  Let $c_0,c_1,c_2>0$ be as in Lemma \ref{lem-coercive}.
  Define a bounded anti-linear form $F$ on $H$ via
  \begin{align*}
  \langle F,v\rangle := \int_\Ome f \big(c_0 - c_1 (r \p_r)^2 + c_2 |\mu| r^2  - \p_\phi^2\big) \OL v \dd x + \int_{\p\Ome} g  \big(c_0 + c_2 r^2 |\mu|\big) \OL v \dd r+ c_1\int_{\p\Ome} (r\p_r+1) g r\p_r \OL v \dd r.
  \end{align*}
  By the definition of $H$ we have $\|F\|_{H'}\lesssim \ZMalp{f}{0}{0}+\ZMalp{f}{0}{-1} + \HMalp{g}{0}{0}+\HMalp{g}{1}{-1}$.
  By Lemma \ref{lem-coercive} and the Lax-Milgram theorem, there is hence a unique $u \in H$ such that for all $v \in H$ we have $B(u,v)  = \langle F,v\rangle$ and $\|u\|_H\lesssim \|F\|_{H'}$.
  For $w \in \cciL{\Ome}$ we solve the test function problem
    in Proposition \ref{prp-testfun} with $\kap:=|\mu|$, i.e.
  \begin{align*} 
    \begin{pdeq}
      \big(c_0 - c_1 (r \p_r)^2 + c_2 |\mu| r^2  - \p_\phi^2\big)  v \ 
      & = \ w, \qquad\qquad &&\text{in $\Ome$},\\
      v \ 
      & = \ 0 \qquad &&\text{on $\p \Ome$}.
    \end{pdeq}
  \end{align*}
  This yields a $v \in \calh$ with $v\restr{\p\Ome}=0$, so that $v\in H$ and thus Lemma \ref{lem-coercive}\eqref{lem-coercive-iii} gives
  \begin{align*}
    \int_{\Ome}(\mu u - \Delta u - f) \OL w \dd x 
    =  0 \qquad
    \forall w \in \cciL{\Ome}.
  \end{align*}
  It follows that $\mu u - \Delta u = f$ in $\Ome$.
  In order to verify the boundary condition, we choose $\eta \in \cciL{\p_0\Ome}$ arbitrary.
  Consider the solution $\rho\in \CRi(\p_0\Ome)$ from Lemma \ref{prp-testfun-bdry} to 
  \begin{align*}
    \big(c_0 - c_1 (r \p_r)^2 + c_2 r^2 |\mu|\big) \rho \ = \ \eta,
  \end{align*}
  and set $v(r,\phi):=\rho(r)\psi(\phi)$ for some $\psi\in \CRi([0,\ang])$ with $1_{[0,\phi']}\le \psi\le 1_{[0,\phi'']}$ for $0<\phi'<\phi''<\ang$.
  Consequently, we have by the definition of $B(u,v)$ and the rapid decay of $v$ towards the tip and due to $\mu u-\Delta u =f$ in $\Ome$, that for all $\eta\in \CRci(\p_0\Ome)$ it holds
  \begin{align*}
  \int_{\pa_0\Ome} ( u + \pa_\nu u - g) \OL \eta \dd r= \int_{\pa_0\Ome} ( u + \pa_\nu u - g) \big(c_0 - c_1 (r \p_r)^2 + c_2 r^2 |\mu|\big) \OL v \dd r = B(u,v)-\langle F,v\rangle=0,
  \end{align*}
  so that $ u + \pa_\nu u = g$ on $\pa_0\Ome$.
  By analogy we also have $ u + \pa_\nu u = g$ on $\pa_1\Ome$.
  Using $-\Delta u=f-u$ and $\p_\nu u=g-u$ we also obtain the additional estimate.
\end{proof}

\subsection{Weighted Estimates}\label{sec:we}

Next, we show that the unique classical solution $u\in H$ from Lemma
\ref{lem-ex-sol} is contained in a weighted space.
We use a negative weight which imposes less control near the tip but more control at infinity.
Recall that the definition of $H$ involves a parameter $\kap>0$.
\begin{lemma}\label{lem-r-weight}
  Let $\kap>0$.
  Then for all $u\in H$ and $\alpha\in [-1,0)$ it holds
    \begin{align*}
       \ZM{u}{0} + \ZM{\nabla u}{0}  + \HM{u}{0} \lesssim_{\kap} \|u\|_H < \ \infty
    \end{align*}
    and
    \begin{align*}    
      |\alp|^2\ZMalp{u}{0}{\alp+1}\le \ZMalp{u}{0}{0}^{-\alp}\ZMalp{\nabla u}{0}{0}^{1+\alp}.
    \end{align*}
\end{lemma}
\begin{proof}
 The estimate $r^{-2\alpha}\leq 1+r^{2}$ gives
 \begin{align*}
\ZM{u}{0} + \ZM{\nabla u}{0} + \HM{u}{0}\lesssim \ZMalp{u}{0}{0} + \ZMalp{ru}{0}{0} + \ZMalp{\nabla u}{0}{0} + \ZMalp{r\nabla u}{0}{0} + \HMalp{u}{0}{0} + \HMalp{ru}{0}{0}&\lesssim_{\kap} \|u\|_H.
\end{align*}
For $\alp=-1$, this is already the complete statement, since $\ZMalp{u}{0}{\alp+1}=\ZM{u}{0}\le \|u\|_H$.

\medskip

If $\alp\in (-1,0)$, we use Lemma \ref{lem-interp}\eqref{it-interp-12-wedge} applied with $\ell=\bet=0$ and $\eta=\alp+1$, so that $|\sclOme{\ell+\eta}|^{2}=\alp^2$, and obtain for all $v\in \CRci(\OL\Ome\BS\set{0})$
\begin{align} \label{rhs-ab100} 
  \alp^4\ZMalp{v}{0}{\alp+1}^2  \  
  \le  \ \ZMalp{v}{0}{0}^{-2\alp} \ZMalp{\nabla v}{0}{0}^{2(1+\alp)}.
\end{align}
By the definition of $H$ there is $\seqN{u}\subset \CRci(\OL\Ome\BS\set{0})$ with $\|u-u_n\|_H \to 0$, in particular $\ZMalp{u-u_n}{0}{0} + \ZMalp{\nabla u-\nabla u_n}{0}{0}\to 0$ as $n\to \infty$, and $u_n\to u$ pointwise almost everywhere.
Using \eqref{rhs-ab100} with $u_n-u_m$, we see that $\seqN{u}$ is Cauchy in the Banach space $\VSooalp{0}{\alp+1}$, and by the pointwise almost everywhere convergence $u_n\to u$, its limit is $u$.
Hence the claimed estimate follows by approximation.
\end{proof}
\begin{lemma}[Weighted Laplace] \label{lem-alptest} 
 Let $\eps\in (0,\pi)$, $\mu\in \Sigma_{\pi-\eps}$, $\alp \in [-1,0]$, and let $u \in H$ with $\ZM{\Delta u}{0}+\HM{\p_\nu u}{0}<\infty$.
 Then it holds 
  \begin{align*}
  \int_{\Ome} r^{-2\alp} f \OL u \dd x+ \int_{\p \Ome} r^{-2\alp} g  \OL u \dd r 
  &=   \mu \ZM{u}{0}^2 + \ZM{\nabla u}{0}^2  
    -    2 \alp^2 \ZMalp{u}{0}{\alp+1}^2  
    + \HM{u}{0}^2,
\end{align*}  
where $f:=\mu u-\Delta u$ and $g:= u + \p_\nu u$.
\end{lemma}
\begin{proof}
  Let $u_n\in \CRci(\OL\Ome\BS\set{0}$ with $\lim_{n\to\infty}\|u-u_n\|_H=0$.
  Then integration by parts yields
\begin{align*}
                 \int_{\Ome} r^{-2\alp} (-\Delta u) \OL {u_n } \dd x
                 &=   \int_{\Ome} \nabla u \cdot \nabla (r^{-2\alp}\OL {u_n}) \dd x 
                 - \int_{\p \Ome} r^{-2\alp} (\p_\nu u)  \OL {u_n} \dd r \\ 
               &=   \int_{\Ome} r^{-2\alp} \nabla u\cdot \nabla\OL{u_n} \dd x 
                 -    2\alp \int_{\Ome} r^{-2\alp-1} u \p_r \OL {u_n} \dd x 
                 + \int_{\p \Ome} r^{-2\alp} (\pa_\nu u)  \OL {u_n} \dd r.
\end{align*}
Letting $n\to\infty$, we may use $\ZM{\Delta u}{0}+\HM{\p_\nu u}{0}<\infty$ to infer
\begin{align*}
                 \int_{\Ome} r^{-2\alp} (-\Delta u) \OL {u } \dd x
               &=   \int_{\Ome} r^{-2\alp} |\nabla u|^2 \dd x 
                 -   2\alp \int_{\Ome} r^{-2\alp-1} u \p_r \OL u \dd x 
                 + \int_{\p \Ome} r^{-2\alp} (\pa_\nu u) \OL u \dd r.
\end{align*}
Additionally, we observe by another approximation (using $\ZMalp{u}{0}{\alp+1}+\ZM{\nabla u}{0}\lesssim \|u\|_H<\infty$)
\begin{align*}
  - 2\alp \int_{\Ome} r^{-2\alp-1} u \p_r \OL  u \dd x 
  &=  - 2\alp \lim_{n\to\infty}\int_{\Ome} r^{-2\alp-1} u_n \p_r \OL  {u_n} \dd x  - \alp \int_0^\ang \int_0^\infty \p_r (r^{-2\alp}  |u_n|^2) \dd r\dd\phi \\
  &=   2 \alp^2 \lim_{n\to\infty}\int_0^\ang \int_0^\infty r^{-2\alp-1} |u_n|^2 \ \dd r\dd\phi \ 
  =   2 \alp^2 \lim_{n\to\infty}\int_\Ome r^{-2\alp-2} |u_n|^2 \dd x\\
  &= 2 \alp^2 \int_\Ome r^{-2\alp-2} |u|^2 \dd x,
\end{align*}
so that the assertion follows upon writing $\pa_\nu u=g-u$ and rearranging the terms.
\end{proof}
From the above lemmas we derive the following estimate on the weighted norms.
\begin{lemma}[Weighted variational solution] \label{lem-weisol}
  Let $\eps\in(0,\pi)$, $\mu\in \Sigma_{\pi-\eps}$, and $\alp \in [-1,0]$.
  Then for any $f \in \cciL{\OL \Ome \BS \{ 0 \}}$ and $g\in \cciL{\p \Ome \BS \{ 0 \}}$ the solution $u \in H$ to \eqref{resolvent} from Lemma \ref{lem-ex-sol} satisfies the estimate
\begin{align*}
                 |\mu| \ZM{u}{0}  +  |\mu|^{\frac12}\ZM{\nabla u}{0} &+  \ZM{\Delta u}{0} +    |\mu|^\frac12 \big(\HM{u}{0} + \HM{ \pa_\nu u}{0}\big) \\
                 & \lesssim_{\eps,\alp}    \ZM{f}{0} + |\mu|^{\frac 14}\HM{ g}{0}    +  |\mu|^{\frac \alp 2} \big(  \ZMalp{f}{0}{0}  +  |\mu|^{\frac 14}\HMalp{  g}{0}{0} \big). 
\end{align*}
For $\resconst>0$ the implicit constant can be chosen uniformly in $|\alp|\ge \resconst$.
\end{lemma}
\begin{proof}
  We note that by Lemma \ref{lem-r-weight} and $\Delta u = \mu u-f$ all terms on the left-hand side of the claimed estimate are finite.
  Testing \eqref{resolvent} with $C_0 |\mu|^\alp \OL u + r^{-2\alp} \OL u$ for some large $C_0 > 0$, we obtain from Lemma \ref{lem-alptest}
\begin{align} \label{comp-prob} 
                 \int_{\Ome} f  (C_0 |\mu|^\alp \OL u + r^{-2\alp} \OL u) \dd x&+ \int_{\p\Ome} g  (C_0 |\mu|^\alp \OL u + r^{-2\alp} \OL u) \dd r = z - 2\alp^2 \ZMalp{u}{0}{\alp+1}^2,
\end{align}
where --- in order to deal with the complexity of the problem and in particular $\mu$ --- we have introduced the complex number
\begin{align*}
  z \ &:= \ C_0 |\mu|^\alp \Big(  \mu \ZMalp{u}{0}{0}^2 + \ZMalp{\nabla u}{0}{0}^2 + \HMalp{u}{0}{0}^2 \Big) + \mu\ZM{u}{0}^2 + \ZM{\nabla u}{0}^2 + \HM{u}{0}^2.
\end{align*}
We have the form $z = \mu a^2 + b^2$ for $a, b \in \R$ and by Lemma
\ref{lem-complane} we get $|z| \gtrsim_\eps(|\mu| a^2 + b^2)$, i.e.
\begin{align*}
  |z| \ &\gtrsim_\eps \ C_0 |\mu|^\alp \Big(  |\mu| \ZMalp{u}{0}{0}^2 + \ZMalp{\nabla u}{0}{0}^2 + \gam\HMalp{u}{0}{0}^2 \Big) + |\mu|\ZM{u}{0}^2 + \ZM{\nabla u}{0}^2 + \HM{u}{0}^2.
\end{align*}
The remaining term on the right hand side of \eqref{comp-prob} can be estimated via Lemma \ref{lem-r-weight} by
\begin{align*} 
  \alp^4\ZMalp{u}{0}{\alp+1}^2  \  
  \le  \ \ZMalp{u}{0}{0}^{-2\alp} \ZMalp{\nabla u}{0}{0}^{2(1+\alp)}
  \le  \ |\mu|^{\alp}  ( |\mu| \ZMalp{u}{0}{0}^2 + \ZMalp{\nabla u}{0}{0}^2).
\end{align*}
For sufficiently large $C_0(\alp,\eps):=\alp^{-2}c_0(\eps)$, this term can be absorbed into $|z|$, and we can estimate the right hand side of \eqref{comp-prob} using the triangle inequality from below.
Applying the Cauchy-Schwarz inequality  and Young's inequality to the left-hand side of \eqref{comp-prob},
we have for $\delta>0$
\begin{align*}
\Big|\int_{\Ome} f  (C_0 |\mu|^\alp \OL u + r^{-2\alp} \OL u)\dd x\Big| &\le C_\delta |\mu|^{-1}(C_0|\mu|^\alp\ZMalp{f}{0}{0}^2 + \ZM{f}{0}^2) + \delta (C_0 |\mu|^{\alp+1} \ZMalp{u}{0}{0}^2 + |\mu| \ZM{u}{0}^2), \\
\Big|\int_{\p\Ome} g  (C_0 |\mu|^\alp \OL u + r^{-2\alp} \OL u) \dd r\Big|&\le C_\delta |\mu|^{-\frac12}(C_0|\mu|^\alp \HMalp{g}{0}{0}^2 + \HM{g}{0}^2) + \delta |\mu|^\frac12(C_0|\mu|^\alp\HMalp{u}{0}{0}^2 + \HM{u}{0}^2) \\
&\lesssim C_\delta |\mu|^{-\frac12}(C_0|\mu|^\alp \HMalp{g}{0}{0}^2 + \HM{g}{0}^2) \\
&\qquad + \delta (C_0|\mu|^\alp(|\mu| \ZMalp{u}{0}{0}^2 + \ZMalp{u}{1}{0}^2) + |\mu| \ZM{u}{0}^2 + \ZM{u}{1}^2),
\end{align*}
where we have used $|\mu|^\frac14\HMalp{u}{0}{\bet}\lesssim |\mu|^\frac12\ZMalp{u}{0}{\bet} + \ZMalp{u}{1}{\bet}$ for $\bet\in \set{\alp,0}$ in view of \eqref{trace-0} in Lemma \ref{lem-trace}.
Absorbing the corresponding solution terms, we obtain
\begin{align*}
  \hspace{6ex} & \hspace{-6ex} 
                C_0 |\mu|^\alp \Big(  |\mu| \ZMalp{u}{0}{0}^2 + \ZMalp{\nabla u}{0}{0}^2 + \HMalp{u}{0}{0}^2 \Big) + |\mu|\ZM{u}{0}^2 + \ZM{\nabla u}{0}^2 + \HM{u}{0}^2 \\ 
  &\lesssim_{\eps,\alp}  \ |\mu|^{-1}(|\mu|^\alp\ZMalp{f}{0}{0}^2 + \ZM{f}{0}^2) + |\mu|^{-\frac12}(|\mu|^\alpha \HMalp{g}{0}{0}^2 + \HM{g}{0}^2). \ 
\end{align*}
In particular, this yields the claimed estimate after multiplying by $|\mu|$ and using the equation in order to get the corresponding control on $\Delta u$ and $|\mu|^{\frac12}\pa_\nu u$ as well.
\end{proof}

\section{Resolvent Problem and Parabolic Equation}\label{sec:res_par}

\subsection{Maximal Regularity for Resolvent Equation}

In this section, we improve the regularity results from Lemma \ref{lem-weisol} iteratively by writing $- \Delta u \ = \ f - \mu u$ and $\pa_\nu u \ = \ g - u$, and using elliptic regularity.
\begin{theorem}[Base regularity for homogeneous norm]\label{thm-res-base-reg} 
  Let $\eps\in(0,\pi)$, $\mu\in \Sigma_{\pi-\eps}$ with $|\mu|\ge 1$, and $\alp \in (-1,0)$.
  Suppose that \eqref{ass-alp} is fulfilled with $\ell=0$.
  Then for $f \in \CRci(\OL\Ome\BS\set{0})$ and $g \in \CRci(\p' \Ome)$, there is $\poly_u\in \VSpol_{2,\alp}$ such that the unique solution $u\in H$ of
  \eqref{resolvent} from Lemma \ref{lem-ex-sol} satisfies
  \begin{align*} 
  |\mu|\ZM{u}{0} + |\mu|^\frac12\ZM{u}{1}+ \ZM{u-\poly_u}{2} + \|\poly_u\|_{\VSpol_{2,\alp}}+|\mu|^{\frac 12}\HM{
      u}{0} \lesssim_{\alp_0,\alp_1,\eps} \ X(\mu),
  \end{align*}
  where
  \begin{align} \label{def-X} %
    X(\mu) \ := \ \ZM{f}{0} + \HM{g}{\frac 12} + |\mu|^{\frac 14}\HM{ g}{0}     +     |\mu|^{\frac \alp 2}(\ZMalp{ f}{0}{0}  +   |\mu|^{\frac14}\HMalp{ g}{0}{0}  ).
  \end{align}
\end{theorem}
\begin{proof}
  By Lemma \ref{lem-weisol} (and since $\ZM{\nabla u}{0}=\ZM{u}{1}$) it suffices to find $\poly_u\in \VSpol_{2,\alp}$ such that $\ZM{u-\poly_u}{2}\lesssim X(\mu)$.
  Since $u$ is a classical solution of the resolvent problem \eqref{resolvent} with Robin boundary conditions, it is also a solution of the elliptic problem
    \begin{align*}
    \begin{pdeq}
          \DeltaB \vvv \ &= \ \t f &&\qquad\text{in} \ \MOme,  \\
          \pa_\nu \vvv \ &= \ \t g &&\qquad\text{on} \ \pa'\Ome, 
    \end{pdeq}
  \end{align*}
  with data $\t f := -\Delta u$ and $\t g := g - u$.
  Observe that
  \begin{align*}
  \ZMS{\t f}{0}{\alp} + \HMS{\t g}{\frac12}{\alp}
  \lesssim_{\rpconst,\resconst} \ \ZMS{\Delta u}{0}{\alp}  + \HMS{g}{\frac12}{\alp} + \HMS{ u}{\frac12}{\alp}.
  \end{align*}
  Since $u\in \VSooalp{1}{\alp}$ by Lemma \ref{lem-r-weight} and Lemma \ref{lem-ZM-dense}, the trace estimate in Lemma \ref{lem-trace} yields $\HMS{u}{\frac12}{\alp}\le (2+(\ang|\alp|)^{-1})\ZM{u}{1}\lesssim_{\resconst} \ZM{u}{1}$, so that
   \begin{align}\label{js502}
  \ZMS{\t f}{0}{\alp} + \HMS{\t g}{\frac12}{\alp}
     &\upref{trace-1}\lesssim_{\rpconst,\resconst} \  \ZMS{\Delta u}{0}{\alp} + \HMS{g}{\frac 12}{\alp}  + \ZM{u}{1} %
     \lesssim_{\rpconst,\resconst,\eps} \  X(\mu),
   \end{align}
   where the last estimate follows by Lemma \ref{lem-weisol}.
  As $u\in H$ and thus in particular $u\in \VSooalp{0}{\bet}$ for all $\bet\in [-1,1)$ by Lemma \ref{lem-r-weight} and Lemma \ref{lem-ZM-dense}, we obtain from Proposition \ref{prp-ellneumann} and $\sclOme{1}=0$ a generalized polynomial\footnote{For the definition of $\ker_N^{\sigma_1,\sigma_2}$ see Definition \ref{def-formal-ker}. The word \textit{generalized} refers to the fact that at this point we have not yet excluded a possible logarithmic contribution.} $\poly_u\in \ker_N^{0,\sclOme{2+\alp}}$ such that $u-\poly_u\in \ZSoo{2}$ and 
  \begin{align*}
  \ZM{u - \poly_u}{2} \upref{est-redsol-2}\lesssim_{\rpconst,\resconst}
  \ZMS{\t f}{0}{\alp} + \HMS{\t g}{\frac12}{\alp}
  \stackrel{\eqref{js502}}{\lesssim}_{\rpconst,\resconst,\eps} \ X(\mu).
  \end{align*}
   Observe that $\poly_u(r,\phi)=a + b\ln r + \mathfrak{q}_u(r,\phi)$ with $\mathfrak{q}_u(r,\phi):=\sum_{\pi_k\in (\tfrac{\pi}{\ang}\Z)\cap (0,\sclOme{2+\alp})} u^{\pi_k}(\phi) r^{\pi_k}$.
   Since $u^{\pi_k}(\phi)=c^{\pi_k}\cos(\pi_k\phi)$ for some constant $c^{\pi_k}\in \C$ by the proof of Proposition \ref{prp-ellneumann}, we have $\|u^{\pi_k}\|_{W^{2,2}((0,\ang))}\lesssim \|u^{\pi_k}\|_{L^2(0,\ang))}<\infty$ and thus $\|\mathfrak{q}_u\|_{\VSpol_{2,\alp}}\lesssim \|\mathfrak{q}_{u}\|_{\VSpol_{0,\alp+2}}<\infty$.
   Before estimating this quantity more precisely, we observe that by
   $\ZMalp{u}{1}{0}<\infty$ and $\alp+1\ge 0$ it holds
   \begin{align*}
   b\int_0^\ang\int_0^1 |r\p_r\ln r|^2 \frac{\dd r}{r} \dd\phi &\lesssim    \int_0^\ang\int_0^1 |r\p_r u|^2 \frac{\dd r}{r} \dd\phi +    \int_0^\ang\int_0^1 r^{-2(\alp+1)} |r\p_r(u-\poly_u)|^2 \frac{\dd r}{r} \dd\phi \\
   & \quad + \int_0^\ang \int_0^1 |\mathfrak{q}_u(r,\phi)|^2 \frac{\dd r}{r} \dd\phi \\
   & \le \ZMalp{u}{1}{0}^2 + \ZMalp{u-\poly_u}{1}{\alp+1}^2 + \|\mathfrak{q}_u\|_{\VSpol_{2,\alp}} \\ &\stackrel{\mathclap{\eqref{wedge-3}}}{\lesssim} \ \ZMalp{u}{1}{0}^2 + \ZMalp{u-\poly_u}{2}{\alp}^2 + \|\mathfrak{q}_u\|_{\VSpol_{2,\alp}}<\infty,
      \end{align*}
   where we have used $\sclOme{\alp+2}\ne 0$ in the application of \eqref{wedge-3}.
   Since $\int_0^\ang\int_0^1 |r\p_r \ln r|^2 \frac{\dd r}{r} \dd\phi =\infty$, this necessitates $b=0$.
     Thus $\poly_u(r,\phi)=a+\mathfrak{q}_u(r,\phi)$.
     Finally, for this polynomial $\poly_u$ we obtain from Lemma \ref{lem-inter1-wedge}
   \begin{align*}
	\|\poly_u\|_{\VSpol_{2,\alp}}&\lesssim \|\poly_{u}\|_{\VSpol_{0,\alp+2}} \lesssim \ZMalp{u}{0}{\alp+1} + \ZMalp{u-\poly_u}{0}{\alp+2} \lesssim \ZM{u}{1} + \ZM{u-\poly_u}{2}\le X(\mu).
	\qedhere
   \end{align*}
 \end{proof}
\begin{corollary}\label{js600}
 In the situation of Theorem \ref{thm-res-base-reg}, the solution $u$ satisfies $u\in \ZS{2}$ and
 \begin{align*}
 \sum_{j=0}^2 |\mu|^{\frac j2}\ZN{u}{2-j}\lesssim_{\rpconst,\resconst,\eps} \ZN{f}{0} + \HN{g}{\frac12} + |\mu|^\frac14\HN{g}{0} + |\mu|^{\frac\alp 2}(\ZNS{f}{0}{0} + |\mu|^\frac14\HNalp{g}{0}{0}).
 \end{align*}
\end{corollary}
\begin{proof}
Let $\poly_u\in \VSpol_{2,\alp}$ be as in Theorem \ref{thm-res-base-reg}.
Since $\poly_u$ contains only terms of scaling between $\sclOme{\alp+1}$ and $\sclOme{\alp+2}$, and since $\supp\zet\subset [0,2]$ and $\supp(1-\zet)\subset [2,\infty)$, we obtain
\begin{align*}
\ZM{\zet\poly_u}{0} + \ZM{\zet\poly_u}{1} + \ZM{(1-\zet)\poly_u}{2} \lesssim \|\poly_u\|_{\VSpol_{2,\alp}}\lesssim X(\mu).
\end{align*}
Thus, writing $u=(u-\zet\poly_u)+\zet\poly_u$, we conclude by Theorem \ref{thm-res-base-reg} 
\begin{align*}
\ZN{u}{2}&\lesssim \ZM{u-\zet\poly_u}{0} + \ZM{u-\zet\poly_u}{1} + \ZM{u-\zet\poly_u}{2} + \|\poly_u\|_{\VSpol_{2,\alp}} \\
&\lesssim \ZM{u}{0} + \ZM{\zet\poly_u}{0} + \ZM{u}{1} + \ZM{\zet\poly_u}{1} + \ZM{u-\poly_u}{2}  + \ZM{(1-\zet)\poly_u}{2} + \|\poly_u\|_{\VSpol_{2,\alp}} \lesssim X(\mu),
\end{align*}
where we have used $|\mu|\ge 1$ in the last step.
Since $|\mu|\ZN{u}{0}=|\mu|\ZM{u}{0}$ and $|\mu|^\frac12\ZN{u}{1}\lesssim |\mu|\ZM{u}{0}+|\mu|^\frac12\ZM{u}{1}$ in virtue of $|\mu|\ge 1$, we obtain by Theorem \ref{thm-res-base-reg}
\begin{align*}
  \sum_{j=0}^2 |\mu|^{\frac j2}\ZN{u}{2-j}\lesssim_{\rpconst,\resconst,\eps} X(\mu).
\end{align*}
This gives the result, since $X(\mu)$ is trivially controlled by the right-hand side of the claimed estimate.
\end{proof}

\begin{proposition}\label{thm-res-high-smooth} %
  Let $\eps\in(0,\pi)$, $\mu\in \Sigma_{\pi-\eps}$ with $|\mu|\ge 1$, and $\alp \in [-1,0]$.
  Suppose $\ell\in \N$ and $(\alp,\ell)$ satisfy \eqref{ass-alp}.
  Let $f \in \CRci(\OL\Ome\BS\set{0})$ and $g \in \CRci(\p'\Ome)$.
  Then  for all $0 \leq m \leq \ell$ there is $\poly_u\in \VSpol_{m+2,\alp}$ such that unique solution $u\in H$ of \eqref{resolvent} from Lemma \ref{lem-weisol} satisfies the estimate
  \begin{align*}
    \ZM{u-\poly_u}{m+2} + \|\poly_u\|_{\VSpol_{m+2,\alp}} \ %
       &\lesssim_{\ang, \rpconst,\resconst,\eps,\ell} \ \sum_{j=0}^{m} |\mu|^{\frac 12(m-j)} \big( \ZMS{f}{j}{\alp}  + \HMS{g}{j + \frac 12}{\alp} \big) \\
       &\qquad \qquad \qquad +  |\mu|^{\frac m2}  \big(|\mu|^{\frac 14}\HM{ g}{0} + |\mu|^{-\frac14}\HM{g}{1}\big). %
     \end{align*}
\end{proposition}
\begin{proof}
 We argue by induction.
 By Theorem \ref{thm-res-base-reg} we obtain a polynomial $\poly_{u,2}\in \ker_N^{0,\sclOme{2+\alp}}$ such that $u-\poly_{u,2}\in \ZSoo{2}$ and the estimate is valid for $m = 0$.
 By application of the elliptic regularity estimate of Proposition \ref{prp-ellneumann} there is $\poly_{u,3}\in \ker_N^{\sclOme{2+\alp},\sclOme{3+\alp}}$ such that $u-\poly_{u,3} \in \ZSoo{3}$ and
    \begin{align*}
       \ZM{u-\poly_{u,2}-\poly_{u,3}}{3} \ %
      &\lesssim_{\ang,\rpconst,\resconst} \ \ZMS{f}{1}{\alp} +   |\mu|\ZMS{u}{1}{\alp} + \HMS{g}{\frac 32}{\alp} +  \HM{u-\poly_{u,2}}{\frac 32} \\ %
      &\lesssim_{\ang,\rpconst,\resconst} \ \ZMS{f}{1}{\alp}  + \HMS{g}{\frac 32}{\alp} +  |\mu|^{\frac 12} X(\mu),
    \end{align*}
    where we have used the trace estimate
    $\HM{u-\poly_{u,2}}{\frac 32} \lesssim_{\ang,\rpconst,\resconst}
    \ZM{u-\poly_{u,2}}{2}$ from Lemma \ref{lem-trace}, and $|\mu|\ge 1$ in the last step.
    By the same argument as in the proof of Theorem \ref{thm-res-base-reg} the polynomial $\poly_{u,2}+\poly_{u,3}$ does not contain a contribution of $\ln r$ or $1$, and is estimated by
    \begin{align*}
    \|\poly_{u,2}+\poly_{u,3}\|_{\VSpol_{3,\alp}}\lesssim \ZM{u}{1} + \ZM{u-\poly_{u,2}-\poly_{u,3}}{3}\le \ZMS{f}{1}{\alp}  + \HMS{g}{\frac 32}{\alp} +  |\mu|^{\frac 12} X(\mu).
    \end{align*}
    Analogously, we get $\poly_{u,4}\in \ker_N^{\sclOme{3+\alp},\sclOme{4+\alp}}$ such that $u-\poly_{u,4} \in \ZSoo{4}$ and
     \begin{align*}
       \ZM{u-\poly_{u,2}-\poly_{u,3}-\poly_{u,4}}{4} \ %
       &\lesssim_{\ang,\rpconst,\resconst} \ \ZMS{f}{2}{\alp} +  |\mu|\ZMS{u-\poly_{u,2}}{2}{\alp} + \HMS{g}{\frac 52}{\alp} + \HM{u-\poly_{u,2}-\poly_{u,3}}{\frac 52} \\
       &\lesssim_{\ang,\rpconst,\resconst} \  \ZMS{f}{1}{\alp}  + \HMS{g}{\frac 32}{\alp} + \ZMS{f}{2}{\alp}  + \HMS{g}{\frac 52}{\alp} +  |\mu|  X(\mu),
     \end{align*}
     as well as
    \begin{align*}
    \|\poly_{u,2}+\poly_{u,3}+\poly_{u,4}\|_{\VSpol_{4,\alp}}&\lesssim_{\ang,\rpconst,\resconst} \ZM{u}{1} + \ZM{u-\poly_{u,2}-\poly_{u,3}-\poly_{u,4}}{4} \\
    &\lesssim_{\ang,\rpconst,\resconst} \ZMS{f}{1}{\alp}  + \HMS{g}{\frac 32}{\alp} + \ZMS{f}{2}{\alp}  + \HMS{g}{\frac 52}{\alp} +  |\mu|  X(\mu).
    \end{align*}
     Iteratively, this yields for $0 \leq m \leq \ell$ the asserted estimate, if one also observes that $|\mu|^{\frac m2}X(\mu)$ is included in the right-hand side, since by Lemma \ref{lem-interp} we have
  \begin{align*}
    |\mu|^{\frac \alp 2}(\ZMalp{ f}{0}{0}  +   |\mu|^{\frac 14}\HMalp{ g}{0}{0}) \  %
    &\lesssim_{\rpconst,\resconst}  \ \ZM{f}{0} +  |\mu|^{-\frac 12}\ZM{ f}{1} +   |\mu|^{\frac 14}\HM{g}{0} + |\mu|^{-\frac14}\HM{g}{1}.
    \qedhere
 \end{align*}
 \end{proof}
 
 \begin{corollary}\label{js601}
 In the situation of Proposition \ref{thm-res-high-smooth}, the solution $u$ satisfies $u\in \ZS{\ell+2}$ and
 \begin{align*}
 \sum_{j=0}^{\ell+2}|\mu|^{\frac j2}\ZN{u}{\ell+2-j}\lesssim_{\rpconst,\resconst,\eps} \sum_{j=0}^\ell |\mu|^{\frac j2}(\ZN{f}{\ell-j}  + \HN{g}{\ell-j+\frac 12}) + |\mu|^{\frac\ell 2}(|\mu|^{\frac14}\HN{g}{0} + |\mu|^{-\frac14}\HN{g}{1}).
 \end{align*}
\end{corollary}
\begin{proof}
The proof is analogous to Corollary \ref{js600}, if one replaces the application of Theorem \ref{thm-res-base-reg} by that of Proposition \ref{thm-res-high-smooth}.
\end{proof}

\begin{proposition}[Polynomial problem]\label{thm-res-pol}
  Let $\eps\in(0,\pi)$, $\mu\in \Sigma_{\pi-\eps}$, and $\alp \in [-1,0]$.
  Suppose $\ell\in \N$ and $(\alp,\ell)$ satisfy \eqref{ass-alp}.
  Let $\poly_f\in \VSpol_{\ell,\alp}$, $\poly_g\in \HSpol_{\ell+\frac12,\alp}$.
  Then there exists a solution $\poly_u\in \VSpol_{\ell+2,\alp}$ to \eqref{resolvent}, and we have
  \begin{align*}
  \sum_{j=0}^{\ell}|\mu|^{\frac j2}\|\poly_u\|_{\VSpol_{\ell-j+2,\alp}}\lesssim_{\rpconst,\resconst,\ell,\ang} \sum_{j=0}^\ell |\mu|^{\frac j2}\big(\|\poly_f\|_{\VSpol_{\ell-j,\alp}}+\|\poly_g\|_{\HSpol_{\ell-j+\frac12,\alp}}\big).
  \end{align*}
\end{proposition}
\begin{proof}
 Since $\frac\pi\ang\notin Q$, we may decompose each $q\in \calq$ uniquely into $n,m\in\N_0$ with $q=n+\frac{\pi}{\ang} m$.
 Matching like terms, we are led for each $q=n+\kappa m< \sclOme{\ell+2+\alp}$ to the problem
 \begin{align*}
 \begin{pdeq}
 (q^2+\p_\phi^2) u^{n,m}(\phi) &= f^{n-2,m}(\phi) - \mu u^{n-2,m}(\phi), \\
 -\p_\phi u^{n,m}(0) & = g^{n-1,m} - u^{n-1,m}(0), \\
 \p_\phi u^{n,m}(\ang) & = g^{n-1,m} - u^{n-1,m}(\ang). \\
 \end{pdeq} 
 \end{align*}
 For $n=0$ all terms on the right-hand side vanish leading to $u^{0,m}(\phi)=0$.
 For $n\ne 0$ we may use Lemma \ref{lem-v-est} with $\lam:=q$, and obtain iteratively the estimate noting that $n +\frac{\pi}{\ang}m\le \scl{\ell-j+\frac12+\alp}$ if and only if $n-1+\frac{\pi}{\ang}m< \sclOme{\ell-j+\alp}$, and that
 \begin{align*}
 |u^{n-1,m}(0)|+|u^{n-1,m}(\ang)|\lesssim_\ang \|u^{n-1,m}\|_{W^{1,2}(0,\ang)}.
 \end{align*}
 Observe that $\|\poly_u\|_{\VSpol_{1,\alp}}=\|\poly_u\|_{\VSpol_{0,\alp}}=0$ due to $\sclOme{1+\alp}<0$, which explains why the sum on the left-hand side of the claimed estimate runs only to $\ell$ instead of $\ell+2$.
\end{proof}

\begin{theorem}[Higher regularity]\label{thm-res-high} %
  Let $\eps\in(0,\pi)$, $\mu\in \Sigma_{\pi-\eps}$ with $|\mu|\ge 1$, and $\alp \in [-1,0]$.
  Suppose $\ell\in \N$ and $(\alp,\ell)$ satisfy \eqref{ass-alp}.
  Let $f \in \ZS{\ell}$ and $g \in \HS{\ell+\frac12}(\p\Ome)$.
  Then there is a unique solution $u\in \ZS{\ell+2}$ of \eqref{resolvent}, and it satisfies the estimate
  \begin{align*}
    \sum_{j=0}^{\ell+2} |\mu|^{\frac j2}\ZN{u}{\ell-j+2} \ %
       &\lesssim_{\ang, \rpconst,\resconst,\eps,\ell} \  \sum_{j=0}^\ell |\mu|^{\frac j2}(\ZN{f}{\ell-j}  + \HN{g}{\ell-j+\frac 12}) + |\mu|^{\frac\ell 2}(|\mu|^{\frac14}\HN{g}{0} + |\mu|^{-\frac14}\HN{g}{1}). %
     \end{align*}
\end{theorem}
\begin{proof}
  Write $f=f_1+\zet\poly_f$ with $f_1\in \ZSoo{\ell}$ and $\poly_f\in \VSpol_{\ell,\alp}$, and similarly $g=g_1+\zet\poly_g$ with $g_1\in \HSoo{\ell+\frac12}$ and $\poly_g\in \HSpol_{\ell+\frac12}$.
  Denote by $\poly_u\in \VSpol_{\ell+2,\alp}$ the solution to \eqref{resolvent} from Proposition \ref{thm-res-pol}.
  Observe that
  \begin{align*}
  \begin{pdeq}
  \mu (\zet\poly_u) - \Delta (\zet\poly_u) &= \zet\poly_f + \mathfrak{q}_f && \text{in } \Ome, \\
  \zet\poly_u + \p_\nu (\zet\poly_u) & = \zet\poly_g && \text{on } \p'\Ome,
    \end{pdeq}
  \end{align*}
  with $\mathfrak{q}_f := - \nabla\zet\nabla\poly_u - (\Delta\zet)\poly_u$.
  From $\supp\nabla\zet\subset [1,2]$ and since the polynomials $\poly_u$ have coefficients which are contained in $W^{\ell+2,2}((0,\ang))$, we obtain that $\mathfrak{q}_f\in \VSoo{\ell}$ with $\sum_{j=0}^\ell |\mu|^{\frac j2}\ZM{\mathfrak{q}_f}{\ell-j}\lesssim \sum_{j=0}^{\ell+2}|\mu|^{\frac j2}\|\poly_u\|_{\VSpol_{\ell-j+2,\alp}}$.
  Now let $u_{\text{reg}}\in \ZS{\ell+2}$ be the solution from Proposition \ref{thm-res-high-smooth} with data $f_1-\mathfrak{q}_f$ and $g_1$.
  By Corollary \ref{js601} we have $u_{\text{reg}}\in \ZS{\ell+2}$, and $u:=u_{\text{reg}}+\zet\poly_u\in \ZS{\ell+2}$ solves \eqref{resolvent}.
  Moreover, the claimed estimate follows from Corollary \ref{js601} and Proposition \ref{thm-res-pol}.
 \end{proof}

 \subsection{Proofs of the Main Theorems}

In this section we give the proof of Theorems \ref{thm-ex} and \ref{thm-higher}.
As noted in Section \ref{sec:int}, we may restrict to $\gam=1$.
\begin{proof}[Proof of Theorem \ref{thm-ex}]
%
%
  We extend $F$ and $G$ to negative times by $0$. For $\mu\in \C$ with $\Re\mu=\bet$ let $\Phi[f,g]$ be the
  solution operator of the resolvent problem from Theorem \ref{thm-res-base-reg}
  with right-hand side $f:= \call F$ and $g:=\call G$.
  We note that $f = f(\mu)$ and $g = g(\mu)$ depend on the parameter $\mu$ and also the solution operator $\Phi = \Phi(\mu)$ of the resolvent problem depends on the (same) parameter $\mu \in \C$, but we will suppress this dependence in our notation.
  Since $|\mu|\geq \bet \geq 1$, we may apply Theorem \ref{thm-res-base-reg} and Corollary \ref{js600} to $u :=\Phi [f,g]$, which yield
  \begin{align*}
                   \sum_{j=0}^2 |\mu|^{\frac j2}\ZN{u}{2-j} +   |\mu|^\frac12\HN{u}{0} \
    &\lesssim_{\alp_0,\alp_1,\ang} \ \ZN{f}{0} + \HN{g}{\frac12} + |\mu|^\frac14\HN{g}{0} + |\mu|^{\frac\alp 2}(\ZNS{f}{0}{0} + |\mu|^\frac14\HNalp{g}{0}{0}).
  \end{align*}
  We define $U$ $:=$ $\call^{-1}_\bet[u]$.
  Then by construction we have $\partial_t U-\Delta U=F$ on $\R\times \Ome$, $U+\pa_\nu U=G$ on $\R\times\p\Ome$, while the Hilbert-space valued Paley-Wiener Theorem \cite[Theorem 1.8.3]{ABH11} shows that $U=0$ for negative times.
  By Plancherel's identity for the Laplace transform (Lemma \ref{lem-laplace}\eqref{lap-plancherel}) we obtain
 \begin{align*}      
                  \|U\|_{\mathbb{E}} &=\sum_{j=0}^2 \NNN{U}{H_{\bet,0}^{\frac j2}(\VS{2-j})} + \NNN{U}{H_{\bet,0}^{\frac12}(\HS{0}(\p'\Ome))}
                  \lesssim \ \sum_{j=0}^2\NNN{|\cdot|^{\frac j2}u}{L^2(S_\bet,\VS{2-j})}+ \NNN{|\cdot|^{\frac12}u}{L^2(S_\bet, \HS{0}(\p'\Ome))}  \\ %
                & \lesssim_{\alp_0,\alp_1,\ang}  \NNN{f}{{L^2(S_\bet,\HS{0})}} + \NNN{g}{L^2(S_\bet, \HS{\frac12}(\p'\Ome))}+ \NNN{|\cdot|^\frac14 g}{L^2(S_\bet, \HS{0}(\p'\Ome))}   \\
                &\qquad \quad \quad  +   \NNN{|\cdot|^{\frac \alp 2} f}{{L^2(S_\bet,L^2(\Ome))}}   +     \NNN{|\cdot|^{\frac14 + \frac \alp 2} g}{L^2(S_\bet, L^2(\p\Ome))}    \ %
                \lesssim \ \|F\|_{\mathbb{F}} + \|G\|_{\mathbb{G}}.
 \end{align*}
 This proves the result.
\end{proof}

\begin{proof}[Proof of Theorem \ref{thm-higher}]
  The proof is analogous as the one before,
  but we replace the application of Theorem \ref{thm-res-base-reg} by the application of Theorem \ref{thm-res-high}, so that we have for $u :=\Phi [f,g]$ we have
     \begin{align*}
                      \sum_{j=0}^{\ell+2}|\mu|^{\frac j2}\ZN{u}{\ell-j+2}  %
       \lesssim_{\alp_0,\alp_1,\ang,\ell} \ \sum_{j=0}^{\ell} |\mu|^{\frac j2} \Big( \ZN{f}{\ell-j}  +  \HN{g}{\ell-j + \frac 12} \Big) +  |\mu|^{\frac \ell 2}  \big(|\mu|^{\frac14}\HN{g}{0} + |\mu|^{-\frac14}\HN{g}{1}\big). %
     \end{align*}
     We define $U %
     := \call^{-1}_\bet[u]$.
     As in the proof of Theorem \ref{thm-ex}, $U$ satisfies the equation, and we have
     \begin{align*}
\|U\|_{\mathbb{E}_{\ell+2}}=\sum_{j=0}^{\ell+2} \NNN{U}{H_{\bet,0}^{\frac j2}(\VS{\ell-j+2})}
		&\lesssim \sum_{j=0}^{\ell+2} \NNN{|\cdot|^{\frac j2}u}{L^2(S_\bet,\ZS{\ell-j+2})} \\ %
       &\lesssim_{\alp_0,\alp_1,\ang,\ell} \ \sum_{j=0}^{\ell} \Big(\NNN{|\cdot|^{\frac j2} f}{L^2(S_\bet,\ZS{\ell-j})}  +  \NNN{|\cdot|^{\frac j2}g}{L^2(S_\bet,\HS{\ell-j+\frac12})} \Big) \\
       &\qquad\qquad\qquad\qquad +  \NNN{|\cdot|^{\frac14} g}{L^2(S_\bet,\HS{0})} + \NNN{|\cdot|^{-\frac14} g}{L^2(S_\bet,\HS{1})}\big) \\
                 &\lesssim %
                     \sum_{j=0}^\ell \Big( \NNN{F}{H_{\bet,0}^{\frac j 2}(\VS{\ell-2})} + \NNN{G}{H_{\bet,0}^{\frac j 2}(\HSL{\ell-j+\frac 12}{\p' \Ome})}\Big) \\
                 &\qquad\qquad\qquad\qquad +  \NNN{G}{H_{\bet,0}^{\frac12(\ell+\frac12)}(\HS{0}(\p'\Ome))} + \NNN{G}{H_{\bet,0}^{\frac12(\ell-\frac12)}(\HS{1}(\p'\Ome))} \\
                 &= \|F\|_{\mathbb{F}_\ell} + \|G\|_{\mathbb{G}_{\ell+\frac12}},
     \end{align*}
     which is the desired bound.
\end{proof}

\appendix

  \section{Integral Transforms}\label{app-a}

  In this section we provide known properties of the Mellin transform and Laplace transform, which constitute an important tool in our analysis.
  It will be convenient to have these transforms defined for Hilbert space valued functions.
  We recall that given $\beta\in\R$, we write $S_\beta$ for the line $\setc{\lambda\in \C}{\Re\lambda=\beta}$, and more generally $S_{(\beta_1,\beta_2)}$ for the strip $\setc{\lambda\in \C}{\Re\lambda\in (\beta_1,\beta_2)}$ if $\beta_1<\beta_2$.
  We recall Bochner spaces of vector-valued integrable functions, and the concept of vector-valued analytic functions, see e.g. \cite[Chapter 1.1 and Appendix A]{ABH11}.
\begin{definition}[Mellin transform] Let $\mathsf{H}$ be a Hilbert space. %
\begin{enumerate}
\item   For $\psi\in \LRloc{1}(\R_+,\mathsf{H})$ and $\lam\in\C$ we define
\begin{align*}
  \calm\psi(\lam) \ := \ \widehat \psi(\lam) \ := \ \frac{1}{\sqrt{2\pi}} \int_0^\infty r^{-\lam} \psi(r) \ \frac{\d r}{r}.
\end{align*}
Whenever this integral converges, we call it the Mellin transform of $\psi$ at $\lam$.
\item  For $\bet\in \R$, $\phi\in \LRloc{1}(S_\bet,\mathsf{H})$ and $r\in \R_+$, we define
 \begin{align}\label{mellin-inv}
\calm_\beta^{-1}\phi(r) \ := \ \frac{1}{\sqrt{2\pi}} \int_{\Re \lam = \bet} r^\lam \phi(\lam) \ \d\Im\lam.
\end{align}
Whenever this integral converges, we call it the inverse Mellin transform of $\phi$ at $r$ (along $S_\bet$).
\end{enumerate}
\end{definition}
For general functions $\psi \in L_{\textrm{loc}}^1(\R_+,\mathsf{H})$, the Mellin transform
might fail to converge for certain $\lam\in\C$. However, if it is well defined
for some $\lam_1,\lam_2\in \C$ with $\Re\lam_1=\bet_1$ and $\Re\lam_2=\bet_2$,
then convergence is also guaranteed on the so called strip of convergence
$S_{(\bet_1,\bet_2)} \SUS \C$. For functions on the wedge $\Ome$ we apply the
Mellin transform in the radial direction and consider the angular variable as a
parameter, \emph{i.e.}, for $\Mpot \in \cciL{\overline{\Ome} \BS \set{0},\mathsf{H}}$ we
write
  \begin{align*}
    \widehat\Mpot(\lam,\phi) \ := \ %
    {\frac{1}{\sqrt{2\pi}}} \int_0^\infty r^{-\lam} \Mpot(r,\phi) \ \frac{\d r}{r} \qquad %
    \text{for $(\lam,\phi) \in \C \times [0,\theta]$.}
  \end{align*}
  The Mellin transform has several
  useful properties which are listed below.
\begin{lemma}[Properties of Mellin transform] \label{lem-mellin} %
  For $\psi \in \cciL{\R_+,\mathsf{H}}$ the Mellin transform $\widehat\psi$ is an entire function.
  Furthermore, we have
  \begin{enumerate}
  \item\label{mel-1}
    $ \widehat{r^\beta \psi}(\lam) \ = \ \widehat\psi\vp{\lam-\beta}$ for all $\beta\in\R$, $\lam\in\C$.
  \item\label{mel-2} $\widehat{r \pa_r \psi}(\lam) \  = \ \lam \widehat\psi\vp{\lam}$.
  \item\label{mel-plancherel} %
    $\displaystyle \int_0^\infty r^{-2\bet} \overline{\psi_1(r)} \psi_2(r) \
    \frac{\d r}{r} \ %
    = \ \int_{\Re \lam=\bet} \overline{\widehat\psi_1(\lam+\gam)}
    \widehat\psi_2(\lam-\gam) \ \d\Im\lam$ for all $\psi_1,\psi_2\in \cciL{\R_+,\mathsf{H}}$ and $\bet,\gam\in \R$. \\
    In particular, for $\bet\in \R$ the Mellin transform can be continuously extended to a linear
    operator
    \begin{align} \label{mellin-iso} %
          \calm_\bet:\big\{ \psi \ : \ r^{-\bet-\frac12} \psi\in L^2(\R_+,\mathsf{H}) \big\} \ \to \ L^2(S_\bet,\mathsf{H}).
        \end{align}
      \item For every $\bet\in\R$ the map \eqref{mellin-iso} is an isometric isomorphism.
      Whenever the integral in \eqref{mellin-inv} converges, it yields the inverse of the map \eqref{mellin-iso}.
\item\label{mel-6} If $\bet_1,\bet_2\in\R$, $\bet_1<\bet_2$, and $r^{\bet_1-\frac12}\psi, r^{\bet_2-\frac12}\psi\in L^2(\R_+,\mathsf{H})$, then $\widehat\psi$ is analytic on the strip $S_{(\bet_1,\bet_2)}$.
  \end{enumerate}
\end{lemma}
\begin{proof}
  We refer to e.g.\@ \cite[Chapter~6]{KozlovMazyaRossmann-Book} in the scalar case.
  The arguments carry over verbatim to the Hilbert space case.
  Parseval's identity in the generalized form can be found e.g.~in \cite[Theorem 73]{Tit86}.
\end{proof}

\medskip

For time--dependent functions we use the Laplace transform.
Since all functions have an extension to negative times by zero, it will be possible to use the two-sided Laplace transform.
In our convention, it is given as follows:
\begin{definition}[Laplace transform] Let $\mathsf{H}$ be a Hilbert space. %
  \begin{enumerate}
  \item For $f\in \LRloc{1}(\R,\mathsf{H})$ and $\mu\in\C$ we define
    \begin{align*}
      \LL f(\mu)  \ := \ \frac 1{\sqrt{2\pi}} \int_{-\infty}^\infty e^{-\mu t} f(t) \dd t.
    \end{align*}
   Whenever this integral converges, we call it the Laplace transform of $f$ at $\mu$. 
  \item For $\bet\in \R$, $g\in \LRloc{1}(S_\bet,\mathsf{H})$ and $r\in \R_+$, we define
    \begin{align}\label{laplace-inv}
      \LL_\bet^{-1} g(t)  \ := \ \frac 1{\sqrt{2\pi}} \int_{\Re \mu = \bet}  e^{\mu t} g(\mu) \dd\Im \mu.
    \end{align}
   Whenever this integral converges, we call it the inverse Laplace transform of $g$ at $t$ (along $S_\beta$). 
  \end{enumerate}
\end{definition}
We note that the Mellin transform is given by the composition of the
Laplace transform and the change of variables $t = \ln r$. In particular, we get
the corresponding properties as for the Mellin transform also for the Laplace
transform if we replace the factor $r^\bet$ by $e^{\bet t}$ and $\frac{\dd r}{r}$ by $\dd t$.
\begin{lemma}[Properties of Laplace transform] \label{lem-laplace} %
  Let $\mathsf{H}$ be a Hilbert space.
  For $\psi \in \cciL{\R,\mathsf{H}}$ the Laplace transform $\widehat\psi$ is an entire function.
  Furthermore, we have
  \begin{enumerate}
  \item\label{lap-1}
    $ \LL(e^{\bet t} \psi)(\mu) \ = \ \LL\psi\vp{\mu-\bet}$ for all $\bet\in\R$.
  \item\label{lap-2} $\LL{\p_t  \psi}(\mu) \  = \ \mu \LL\psi\vp{\mu}$.
  \item\label{lap-plancherel} %
    $\displaystyle \int_{-\infty}^\infty e^{-2\bet t} \overline{\psi_1(t)} \psi_2(t)\dd t  %
    = \int_{\Re \mu=\bet} \overline{\LL\psi_1(\mu+\gam)} \LL\psi_2(\mu-\gam) \dd\Im\mu$ for all $\psi_1, \psi_2\in \cciL{\R,\mathsf{H}}$ and $\bet,\gam\in \R$. \\
    In particular, for $\bet\in \R$ the Laplace transform can be continuously extended to a linear
    operator
    \begin{align} \label{laplace-iso} %
          \call_\bet:\big\{ \psi \ : \ e^{-\bet t} \psi\in L^2(\R,\mathsf{H}) \big\} \ \to \ L^2(S_\bet,\mathsf{H}).
        \end{align}
      \item For every $\bet\in\R$ the map \eqref{laplace-iso} is an isometric isomorphism.
      Whenever the integral in \eqref{laplace-inv} converges, it yields the inverse of the map \eqref{laplace-iso}.
\item\label{lap-6} If $\bet_1,\bet_2\in\R$, $\bet_1<\bet_2$, and $e^{-\bet_1t}\psi, e^{-\bet_2t}\psi\in L^2(\R,\mathsf{H})$, then $\LL\psi$ is analytic on the strip $S_{(\bet_1,\bet_2)}$.
  \end{enumerate}
\end{lemma}

\section{Neumann Problem on the Wedge}\label{sec:neu}

In this section, we consider the elliptic boundary problem with Neumann conditions, i.e.
  \begin{align}\label{ell-neu}
    \begin{pdeq}
          \DeltaB \vvv \ &= \ \fff &&\qquad\text{in} \ \MOme,  \\
          \pa_\nu \vvv \ &= \ \ggg &&\qquad\text{on} \ \pa'\Ome. 
    \end{pdeq}
  \end{align}
Observe that the direction of the outer normal vector yields $\p_\nu=-\frac1r\p_\phi$ on $\p_0\Ome$ and $\p_\nu=\frac1r\p_\phi$ on $\p_1\Ome$.
We note that elliptic problems of the form \eqref{ell-neu} have been studied in the literature, see e.g.~\cite{BGKMRS24,KozlovMazyaRossmann-Book,KSW24}.
Since we treat higher regularity beyond the regime of the first resonance, we include details here for the convenience of the reader.
We first give a solution of the corresponding system to \eqref{ell-neu} in Mellin variables.
\begin{lemma}[Neumann problem in Mellin variables] \label{lem-v-est} %
  Let $\ang \in (0,2\pi)$, $\ffff\in\CRi([0,\ang])$, and
  $\gggg_1,\gggg_2\in\C$.  Then for any $\lam \in \C$ with
  $\lam \notin \frac{\pi}{\ang} \Z$, there is a unique classical solution
  $\vv(\lam,\cdot)$ to the boundary-value problem
\begin{subequations} \label{sys-gen}
  \begin{align} %
    (\lam^2  + \p_\phi^2) \vv(\lam,\phi) \ &= \ \ffff(\phi) \qquad \text{for } \phi \in (0,\ang), \label{sys-gen-uphi} \\
    -\p_\phi \vv(\lam,0) \ &= \ \gggg_1, \label{sys-gen-navier} \\ %
    \p_\phi\vv(\lam,\ang) \ &= \ \gggg_2. \label{sys-gen-stress}
  \end{align}
\end{subequations}
The function $\vv(\cdot,\phi):\C\to\C$ is meromorphic for each
$\phi\in [0,\ang]$ with all poles contained in $\frac{\pi}{\ang} \Z$.
The pole at $\lam=0$ is at most of order $2$ and all other poles are simple.
The solution $\vv$ can be represented as
\begin{subequations}\label{v-repr}
\begin{align}\label{lem-v-repr-1}
\vv(\lam,\phi) \ = \ -G(\lam,\phi,0)\gggg_1 + G(\lam,\phi,\ang)\gggg_2 + \int_0^\ang G(\lam,\phi,\phi') \ffff(\phi') \ \d\phi',
\end{align}
where the meromorphic Green's function $G(\cdot,\phi,\phi')$ to \eqref{sys-gen}
is  given by
\begin{align}\label{lem-v-repr-3}
  G(\lam,\phi,\phi') = \frac{1}{\lam\sin(\lam\ang)}
  \begin{TC}
    \cos(\lam(\ang-\phi'))\cos (\lam \phi) &\text{for $\phi \in [0,{\phi'}]$}, \\
    \cos(\lam\phi')\cos (\lam(\ang-\phi)) &\text{for $\phi \in ({\phi'},\ang]$}.
  \end{TC}  
\end{align}
\end{subequations}
Furthermore, if $\rpconst,\resconst>0$ and $\ell \in \N_0$, then whenever $\ang\labs{\Re \lam}\leq \rpconst$ and $\dist(\ang\lam, \pi\Z)\geq \resconst$, we have
  \begin{align} %
    \hspace{6ex} & \hspace{-6ex} %
                   \sum_{j+m=\ell+2}  \int_0^\ang  |\lam|^{2j}\labs{\p_\phi^m \vv(\lam,\phi)}^2 \d\phi  \lesssim_{\rpconst,\resconst,\ell} \ \sum_{j+m=\ell} \int_0^\ang  |\lam|^{2j}|\pa_\phi^m \ffff(\phi)|^2 \d\phi +  |\lam|^{2\ell + 1}(|\gggg_1|^2 +|\gggg_2|^2).  \label{sys-gen-est}
\end{align}
\end{lemma}
\begin{proof}
  The claim about the pole set follows directly from the formula for $G$.
  For $\lam\ang\notin \pi\Z$ it follows from standard ODE arguments that $\vv(\lam,\cdot)$ is the unique solution to \eqref{sys-gen}.
If $\ffff=0$, \eqref{sys-gen-est} follows from the representation
of $\vv$ and Lemma~\ref{lem-aux1}. In the following, we hence assume
$\gggg_1 = \gggg_2 = 0$. For $\gggg_1 = \gggg_2 = 0$, we test \eqref{sys-gen-uphi} with
$\overline{ \vv}$ to get
\begin{align*}
  \int_0^\ang\overline{ \vv} \ffff  \ \d\phi \ %
  &= \ \int_0^\ang \overline{ \vv} \p_\phi^2 \vv  \ \d\phi  + \lam^2 \int_0^\ang \labs{\vv}^2 \ \d\phi  \ %
  = \ - \int_0^\ang \labs{\p_\phi \vv}^2 \d\phi +  \lam^2 \int_0^\ang \labs{\vv}^2 \ \d\phi.
\end{align*}
We take the real part and absorb the term on the left hand side using Young's
inequality.  If $\lam$ has a large imaginary part
$\ang\labs{\Im \lam} \geq 2\rpconst$, we have $\Re (\lam^2) \sim - \labs{\lam}^2$,
which implies
\begin{align}\label{sys-gen-large-freq}
  \int_0^\ang  \left( \labs{\lam}^2 \labs{\vv}^2 + \labs{\p_\phi \vv}^2 \right) \d\phi \ %
  \lesssim \ \labs{\lam}^{-2} \int_0^\ang \labs{\ffff}^2 \d\phi.
\end{align}
If $\ang\labs{\Im\lam}\leq 2\rpconst$ (and hence $\ang\labs{\lam}\leq 3\rpconst\lesssim 1$), we have $\labs{\lam G(\lam,\cdot,\phi')}$, $\labs{\pa_\phi G(\lam,\cdot,\phi')} \lesssim 1$ since the sine and cosine are bounded on $B_{(\ang-\phi')\labs{\lam}},B_{\phi'\labs{\lam}}\subset B_{3\rpconst}\subset \C$ and since by the assumptions on $\lam$
\begin{align*}
  |\sin(\ang\lam)|^2 \ = \ \sin^2(\ang\Re \lam) + \sinh^2(\ang\Im\lam) \geq \sin^2(\ang\Re \lam)\gtrsim_{\alp_1} \ 1.
\end{align*}
With Jensen's inequality we estimate
\begin{align}\label{sys-gen-low-freq}
  \int_0^\ang \left( \labs{\lam}^2 \labs{\vv}^2 + \labs{\p_\phi \vv}^2 \right) \d\phi \ %
  \upref{lem-v-repr-3}\lesssim  \ \ang \Big(\int_0^\ang \labs{\ffff} \d\phi\Big)^2 \ %
  \leq \ \ang^2 \int_0^\ang \labs{\ffff}^2 \d\phi \lesssim \labs{\lam}^{-2} \int_0^\ang \labs{\ffff}^2 \d\phi,
\end{align}
where we have used $\ang\labs{\lam}\lesssim 1$ in the last step.
By multiplying \eqref{sys-gen-large-freq} and \eqref{sys-gen-low-freq} with $\labs\lam^2$ and using equation~\eqref{sys-gen-uphi} once more, we obtain the corresponding bound on $\pa_\phi^2 \vv$, thus proving \eqref{sys-gen-est} for $\ell = 0$.

\medskip

Now assume that the assertion holds for $\ell\in \N_0$.
Multiplying \eqref{sys-gen} by $\lam$, we obtain from \eqref{sys-gen-est}
\begin{align*} %
  & \sum_{j+m=\ell+2} \int_0^\ang  \labs{\lam^{j+1} \pa_\phi^m \vv(\lam,\phi)}^2 \d\phi  \ %
    \lesssim \ \sum_{j+m=\ell} \int_0^\ang |\lam^{j+1}\pa_\phi^m \ffff(\phi)|^2 \d\phi +  |\lam^{\ell+1 + \frac 12}\gggg_1|^2 +  |\lam^{\ell+1 + \frac 12} \gggg_2|^2. 
\end{align*}
Using
$\pa_\phi^{\ell+3}\vv=-\lam^2\pa_\phi^{\ell+1}\vv + \pa_\phi^{\ell+1} \ffff$ by
virtue of \eqref{sys-gen-uphi}, we obtain the assertion for $\ell+1$ and can
conclude by induction. 
\end{proof}
The above solution gives us information about the kernel of the Laplace operator with Neumann boundary condition.
\begin{definition}[Formal kernel]\label{def-formal-ker} %
  For $\ang \in (0,2\pi)$ and $k\in\Z$, we write
  \begin{align*}
    \pi_k:=\frac{k\pi}{\ang}
  \end{align*}
  and define the kernel of the Laplace operator $\Delta$ for the Neumann problem \eqref{ell-neu} by
  \begin{align*}
    \kerneu \ 
    &:= \ \spann \Big \langle \set{\ln r}  \cup  \setc{r^{\pi_k} \cos(\pi_k \phi)}{k\in \Z} \Big \rangle \ \in \ \spann \big\langle \ln r \big\rangle \oplus \PSOme.
  \end{align*}
  For $\sig_1,\sig_2\in \R$ we define the kernel of limited scaling width by
  \begin{align*}
    \kerneun{\sig_1}{\sig_2} \ 
    &:= \ \spann \Big \langle \set{1,\ln r} \cup \setc{r^{\pi_k} \cos(\pi_k \phi)}{\pi_k\in [\sig_1,\sig_2]\cup [\sig_2,\sig_1]} \Big \rangle.
  \end{align*}
\end{definition}
\begin{proposition}[Elliptic Neumann problem] \label{prp-ellneumann} %
  Let $\ang \in (0,2\pi)$, $\fff \in \CRci(\overline{\Ome}\BS\set{0})$
  and $\ggg\in \CRci(\p'\Ome)$.
  Then the following assertions hold.
  \begin{enumerate}
  \item\label{prp-ellneumanni} Let $(\ell,\alp) \in \N_0 \times \R$ fulfill $\sclOme{\ell+\alp+2}\notin \frac{\pi}{\ang}\Z$.
  Then there exists a classical solution $\vvv\in \ZSoo{\ell+2}$ to \eqref{ell-neu} with $\sum_{j=0}^\ell \ZMM{\vvv}{\ell-j+2}{\alp+j}<\infty$.
  \item\label{prp-ellneumannii} Let $(\ell_1,\bet_1), (\ell_2,\bet_2) \in \N_0 \times \R$ fulfill $\sigma_j:=\sclOme{\ell_j+\bet_j+2}\notin \frac{\pi}{\ang}\Z$ for $j\in\set{1,2}$.
  Then for two classical solutions $v_1$, $v_2$ to \eqref{ell-neu} with $\ZMalp{v_1}{\ell_1+2}{\bet_1}<\infty$ and $\ZMalp{v_2}{\ell_2+2}{\bet_2}<\infty$ we have $v_1 - v_2 \in \kerneun{\sig_1}{\sig_2}$.
  \item\label{prp-ellneumanniii} Let $\rpconst,\resconst>0$.
  If $(\ell,\alp)\in \N_0\times \R$ fulfills $\dist(\ang\sclOme{\ell+\alp+2},\pi\Z)\ge \resconst$, $\ang|\sclOme{\ell+\alp+1}|\ge \resconst$ (and $\ang|\sclOme{\ell+\alp}|\ge\resconst$ if $\ell>0$) as well as $\ang|\sclOme{\ell+\alp+2}|\le \rpconst$, then we have
 \begin{align}\label{est-redsol-2}
   \ZM{\vvv}{\ell+2} &\lesssim_{\rpconst,\resconst,\ell} \ \ZMS{\fff}{\ell}{\alp} + \HMS{\ggg}{\ell+\frac12}{\alp}.
 \end{align}
 \item\label{prp-ellneumanniv} Let $\rpconst,\resconst>0$, $\bet_1,\bet_2\in \R$ and $\vartheta\in (0,1)$ with $\dist(\vartheta,\set{0,1})>\resconst$.
 Write $\bet:=(1-\vartheta)\bet_1+\vartheta\bet_2$.
 Let $\ell_1,\ell_2\in \N_0$ fulfill $\dist(\ang\sclOme{\ell_j+\bet_j+2},\pi\Z)\ge \resconst$, $\ang|\sclOme{\ell_j+\bet_j+1}|\ge \resconst$ (and $\ang|\sclOme{\ell_j+\bet_j}|\ge\resconst$ if $\ell_j>0$) as well as $\ang|\sclOme{\ell_j+\bet_j+2}|\le \rpconst$ for $j\in \set{1,2}$, and let $v_1\in \ZSooalp{\ell_1+2}{\bet_1}$, $v_2\in \ZSooalp{\ell_2+2}{\bet_2}$ be the corresponding classical solutions to \eqref{ell-neu}.
 Define $\ell:=(1-\vartheta)\ell_1+\vartheta\ell_2$.
 If $\sigma_1<\sigma_2$ for $\sigma_j:=\sclOme{\ell_j+\bet_j+2}$, and if $v_1-v_2\in \ker_N^{\sigma_1,\sigma_2}$ does not contain a contribution from $\langle \set{1,\ln r}\rangle$, then it holds
  \begin{align*}
   \NNN{v_1-v_2}{\VSpol_{\ell+2,\bet}}\lesssim_{\rpconst,\resconst,\ell,\ang} \sum_{j=1}^2 \big(\ZMS{\fff}{\ell_j}{\bet_j} + \HMS{\ggg}{\ell_j+\frac12}{\bet_j}\big).
  \end{align*}
  \end{enumerate}
\end{proposition}
\begin{proof}
  \textit{(i):} For fixed $\ell\in\N_0$ and $\alp\in\R$ with $\ang(\ell+\alp+1)\notin \pi\Z$ we define $\vvv: (0,\infty)\times [0,\ang] \to \R$ via
\begin{align}\label{lem-redsol-def}
\vvv(r,\phi):= \frac{1}{\sqrt{2\pi}}\int_{\Re\lam=\ell+\alp+1} r^\lam \vv(\lam,\phi)\,\d\Im\lam,
\end{align}
where $\vv(\lam,\cdot)$ is given by \eqref{lem-v-repr-1} with data
$\ffff:=\widehat{r^2\fff}(\lam,\cdot)$, $\gggg_1:=\widehat{(r\ggg)}(\lam,0)$, and
$\gggg_2:=\widehat{(r\ggg)}(\lam,\ang)$.
Observe that for all $\phi\in [0,\ang]$,
$\vv(\cdot,\phi)$ is meromorphic with its only poles in $\tfrac{\pi}{\ang} \Z$,
and that for $\Rel\in\R$ and $m\in \N_0$, there are $M,\eps,\delta\in(0,\infty)$
such that for all $\lam\in S_{[-\Rel,\Rel]}$ and
$\labs{\Im\lam}\geq M$ it holds
\begin{align}\label{sys-gen-im-est}
\labs{\pa_\phi^m\vv(\lam,\phi)}\leq \delta e^{-\eps\labs{\Im\lam}}.
\end{align}
Indeed, for all $m\in\N_0$ and $\phi\in(0,\ang)$ the functions $\pa_\phi^m \widehat{r^2\fff}(\cdot,\phi)$, $\widehat{(r\ggg)}(\cdot,0)$, and $\widehat{(r\ggg)}(\cdot,\ang)$ are analytic with exponential decay as $\labs{\Im\lam}\to\infty$ as in \eqref{sys-gen-im-est}.
Since $\labs{\sin(\lam\phi)}\sim\labs{\cos(\lam\phi)}\sim e^{\phi\labs{\Im\lam}}$ as $\labs{\Im\lam}\to\infty$, the Green's function $G(\lam,\phi,\phi')$ is bounded on the set $\setc{\lam\in S_{[-\Rel,\Rel]}}{\labs{\Im\lam}\geq M}$, so that the exponential decay is transferred to $\lam \mapsto \pa_\phi^m\vv(\lam,\phi)$ as claimed.
By Plancherel's theorem in form of Lemma \ref{lem-mellin}\eqref{mel-plancherel} and the exponential decay of $\lam \mapsto r^\lam\vv(\lam,\phi)$ on $\Re\lam=\ell+\alp+1$, we obtain that all derivatives $(r\pa_r)^j\pa_\phi^m\vvv$ are contained in $L^2_{\text{loc}}(\overline{\Omega}\setminus\set{0})$.
Hence, $\vvv$ is smooth.
Moreover,
by \eqref{mellin-inv} in Lemma \ref{lem-mellin} it holds $\widehat\vvv(\lam,\phi)=\vv(\lam,\phi)$ for all $\phi\in[0,\ang]$ and all $\lam\in S_{\ell+\alp+1}$.

\medskip

Next we verify that $v$ does indeed solve \eqref{ell-neu}.
Note that by Lemma \ref{lem-v-est}, we have for all $\lam\in S_{\ell+\alp+1}$
\begin{align*}
  \left(\lam^2 + \pa_\phi^2\right) \widehat\vvv(\lam,\phi) \ = \ \widehat{(r^2\fff)}(\lam,\phi) \quad \text{for} \quad \phi \in (0,\ang),
\end{align*}
as well as $-\pa_\phi\widehat\vvv(\lam,0) = \widehat{(r\ggg)}(\lam,0)$ and $\widehat\vvv(\lam,\ang) = \widehat{\ggg}(\lam,\ang)$.
Since the Mellin transform induces an isomorphism between $\{u: r^{\ell+\alp+\frac12} u\in L^2(\R_+,\C)\}$ and $L^2(S_{\ell+\alp+1},\C)$, cf.~Lemma \ref{lem-mellin}\eqref{mel-plancherel}, we conclude that $\vvv$ solves problem~\eqref{ell-neu} by Lemma \ref{lem-mellin}\eqref{mel-2}.
This proves that $v$ is a classical solution to \eqref{ell-neu}.

\medskip

For $\Re\lam = \sclOme{\ell+\alp+2} = \scl{\ell+\alp+\frac32}=\ell+\alp+1$, the
assumptions on $\lam$ in Lemma \ref{lem-v-est} are fulfilled, so that from
\eqref{sys-gen-est} and the Mellin representation of the norms we get for all $\rpconst,\resconst>0$ with $\dist(\ang\sclOme{\ell+\alp+2},\pi\Z)\ge \resconst$ and $\ang|\sclOme{\ell+\alp+2}|\le \rpconst$, and for all $0\le j\le \ell$ that
\begin{align} \label{est-redsol} %
  \ZMM{\vvv}{\ell -j + 2}{\alp+j} \ &\lesssim_{\rpconst,\resconst,\ell} \
                          \ZMS{r^{2}\fff}{\ell-j}{\alp+j+2} +
                          \HMS{r\ggg}{\ell-j+\frac12}{\alp+j+1}<\infty.
\end{align}
In particular $\vvv\in \HSoo{\ell+2}$ by Lemma \ref{lem-ZM-dense}.

\medskip

\textit{(ii):} By classical methods two classical solutions $\vvv_1$ and
$\vvv_2$ with at most polynomial growth differ only by elements in $\kerneu$,
cf.~\cite{Mar17,Wil68}.  Suppose first that $\sig_1=\sig_2$.  Since
$\ZMalp{v}{\ell_1+2}{\bet_1}=\infty$ for all
$v$ in the span of $\setc{r^{\pi_k}\cos(\pi_k\varphi)}{k\in\Z\setminus\set{0}}$ and
$\ZMalp{v}{\ell_1+2}{\bet_1}=0$ for
$\kerneun{\sig_1}{\sig_1}=\spann\big\langle\set{1,\ln r}\big\rangle$ (the set
equality being a direct consequence of \eqref{ass-alp}), we obtain the
result for $\sig_1=\sig_2$.

\medskip

Suppose now $\sig_1\ne \sig_2$.
By the result for $\sig_1=\sig_2$, $\vvv_1$ and $\vvv_2$ are uniquely determined up to elements in $\spann\big\langle\set{1,\ln r}\big\rangle$.
We can hence assume that $\vvv_1$ and $\vvv_2$ are given via \eqref{lem-redsol-def} corresponding to $(\ell_1, \bet_1)$ and $(\ell_2, \bet_2)$, respectively.
Without loss of generality we assume $\ell_1+\bet_1 \ge \ell_2+\bet_2$.
Since $\lam \mapsto r^\lam\vv(\lam,\phi)$ is meromorphic with the uniform exponential decay \eqref{sys-gen-im-est}, the values of $\vvv_1(r,\phi)$ and $\vvv_2(r,\phi)$ differ by the sum of the residues of $\psi(\lam):= r^\lam\vv(\lam,\phi)$ evaluated at the poles $\pi_k=\frac{k\pi}{\ang}$, $k\in\Z$ which lie between $\sig_1$ and $\sig_2$, that is
\begin{align}\label{lem-redsol-res-cond3}
\vvv_1(r,\phi) - \vvv_2(r,\phi) = \sum_{\pi_k\in (\frac{\pi}{\ang}\Z)\cap (\sig_2,\sig_1)} \text{Res}_\psi(\pi_k).
\end{align}
Since all poles at $\lam\ne 0$ are simple, the residue for $k\ne 0$ is calculated with help of
\begin{align*}
\cos(\pi_k(\ang-\phi)) = (-1)^k \cos(\pi_k\phi), \quad \text{Res}_{1/\sin(\lam\ang)}(\pi_k)=\frac{1}{\ang\cos(\pi_k\ang)} = \frac{(-1)^{k}}{\ang}
\end{align*}
via the definition of $\vv$ in \eqref{lem-v-repr-1} and \eqref{lem-v-repr-3} as
\begin{align*}
  \hspace{2ex} & \hspace{-2ex} %
\text{Res}_\psi (\pi_k) = r^{\pi_k}\text{Res}_{\vv(\cdot,\phi)}(\pi_k) 
= \frac{r^{\pi_k}\cos(\pi_k\phi)}{\pi_k \ang} \Big[-\gggg_1(\pi_k)+(-1)^k\gggg_2(\pi_k)+\int_0^\ang \cos(\pi_k\phi')\ffff(\pi_k,\phi')\d\phi'\Big].
\end{align*}
For $k=0$, we observe that $G(\cdot,\phi,\phi')$ is even in $\lam$ and
holomorphic away from $\lam=0$ so that
\begin{align*}
  \text{Res}_{G(\cdot,\phi,\phi')}(0)=0.
\end{align*}
Since $\text{Res}_{uv}(0)=\text{Res}_u(0) v(0)+ (\lam^2 u(\lam) \pa_\lam)|_{\lam=0} v$ if $u$ possesses a pole of order at most two at $\lam=0$ and $v$ is holomorphic in $0$, we have
\begin{align*}
  \text{Res}_\psi (0) \ %
  &=\ (\lam^2 G(\lam,\phi,0)  \pa_\lam)|_{\lam=0}(-r^\lam \gggg_1(\lam)) + (\lam^2 G(\lam,\phi,\ang)  \pa_\lam)|_{\lam=0}(r^\lam \gggg_2(\lam))  \\
  &\qquad +\int_0^\ang (\lam^2 G(\lam,\phi,\phi')  \pa_\lam)|_{\lam=0}(r^\lam \ffff(\lam,\phi')\d\phi') \\
  &= \ \frac{1}{\ang}(\ln r +\p_\lam) \Big[-\gggg_1(0)+\gggg_2(0)+\int_0^\ang \ffff(0,\phi')\d\phi'\Big].
\end{align*}
Thus \eqref{lem-redsol-res-cond3} gives $v_1 - v_2 \in \kerneun{\sig_1}{\sig_2}$.

\medskip

\textit{(iii):} 
Estimate \eqref{est-redsol-2} follows from \eqref{est-redsol}, Lemma~\ref{lem-hardy-mel} and
Lemma~\ref{lem_mellin_domain}: Indeed, applying Lemma~\ref{lem_mellin_domain}
with $\beta=2$, we have that
\begin{align*}
\ZMS{r^{2}\fff}{\ell}{\alp+2} \stackrel{\eqref{wedge-1}}{\le} \sum_{j=0}^{\ell}\max\Big\{\left|\tfrac{\sclOme{\ell+\alp+2}}{\sclOme{\ell+\alp}}\right|^j,1\Big\}\ZMS{\fff}{\ell}{\alp} \lesssim_{\rpconst,\resconst,\ell} \ZMS{\fff}{\ell}{\alp}.
\end{align*}
On the other hand, applying Lemma~\ref{lem-hardy-mel} with $\ell$ replaced by $\ell+\frac12$ and $\beta=1$, we have with $\scl{\ell+\frac12+\alp+\beta}=\sclOme{\ell+\alp+2}$ and $\scl{\ell+\frac12+\alp}=\sclOme{\ell+\alp+1}$ that
\begin{align*}
\HMS{r\ggg}{\ell+\frac12}{\alp+1} \stackrel{\eqref{hardy-3}}{\le} \max\Big\{\left|\tfrac{\sclOme{\ell+\alp+2}}{\sclOme{\ell+\alp+1}}\right|^{\ell+\frac12},1\Big\} \HMS{\ggg}{\ell+\frac12}{\alp} &\lesssim_{\rpconst,\resconst,\ell} \HMS{\ggg}{\ell+\frac12}{\alp}.
\end{align*}
\textit{(iv):}
Follows from \eqref{lem-redsol-res-cond3} and the representation of $\text{Res}_{\psi}(\pi_k)$ for $k\ne 0$ in the proof of part $(ii)$, if one observes Lemma \ref{lem-inter1-wedge} below.
\end{proof}
Finally we give a lemma which is a multiplicative variant of a corresponding lemma in \cite{GiacomelliGnannKnuepferOtto-2014}.
\begin{lemma} \label{lem-inter1-wedge} Let $\bet_1 < \bet < \bet_2$. Then there
  is $C_\bet < \infty$, such that for any $v \in L_{\rm loc}^1(\Ome)$ and
  $\ccc \in L_{\rm loc}^1((0,\ang))$
  \begin{align} \label{est-inhom} %
        \NTL{\ccc}{(0,\ang)} \ %
        &\lesssim \  C_\bet\ZMalp{\vvv}{0}{\bet_1}^{\frac{\bet_2-\bet}{\bet_2-\bet_1}} \ZMalp{\vvv-c r^\bet}{0}{\bet_2}^{\frac{\bet-\bet_1}{\bet_2-\bet_1}}, %
      \end{align}
      as long as both factors on the right hand side are finite.
    \end{lemma}
    \begin{proof}
  We may assume $\ZMalp{\vvv}{0}{\bet_1}, \ZMalp{\vvv-c r^\bet}{0}{\bet_2}\in (0,\infty)$.
  Let $R>0$ and $\Ome_R = (\frac12 R,R) \times (0,\ang)$. Then
  \begin{align} \label{lap-wedge} %
    \NTL{\ccc}{(0,\ang)}^2 \
    &\lesssim_\bet \ R^{-2\bet} \int_{\Ome_R} |\ccc r^\bet|^2 \ \frac{dr}r d\phi \
    \lesssim_\bet  \  R^{-2\bet} \int_{\Ome_R} |\vvv|^2 \ \frac{dr}r d\phi
      + C_\bet R^{-2\bet} \int_{\Ome_R} |\vvv-\ccc r^\bet|^2 \ \frac{dr}r d\phi  \\
    &\lesssim_\bet \  R^{2(\bet_1-\bet)} \ZMalp{\vvv}{0}{\bet_1}^2 + C_\bet R^{2(\bet_2-\bet)} \ZMalp{\vvv-cr^\beta}{0}{\bet_2}^2. \notag
  \end{align}
  Estimate \eqref{est-inhom} follows by minimizing the right hand side in $R$, i.e.~with $R^{\bet_2-\bet_1}$ \ $:=$ $\tfrac{\ZMalp{\vvv}{0}{\bet_1}}{\ZMalp{\vvv-cr^\beta}{0}{\bet_2}}$. %
\qedhere
\end{proof}

\section{Auxiliary Estimates} %

\begin{lemma}[Auxiliary estimate] \label{lem-aux1} %
  Let $\rpconst>0$, $\resconst\in(0,\frac{\pi}{2}]$. Then there is $c>0$ such that for all $\ang>0$, for all $f, g \in \{ \sin, \cos \}$, and for all $\lam \in \C$ with $\ang\labs{\Re\lam}\leq \rpconst$ and
  \begin{align*}
     \dist(\ang|\Re \lam|, g^{-1}(\{0\}))\geq \resconst
  \end{align*}
  we have
  \begin{align}\label{eq-f-cosh} 
  0<\sin(\resconst)\leq \labs{g(\lam\ang)} \leq \cosh(\ang\Im\lam)
  \end{align}
  and
  \begin{align} \label{int-coscos} %
    |\lam| \int_0^\ang \frac{|f(\lam \phi)|^2}{|g(\lam \ang)|^2} \ \d\phi \ %
    \leq \   \max\left\{\frac{2\rpconst \cosh^2(\rpconst)}{\sin^2(\resconst)}, \frac{\rpconst+\sinh(\rpconst)\cosh(\rpconst)}{\sinh^2{\rpconst}}\right\}.
  \end{align}
\end{lemma}
\begin{proof}
  By a straightforward calculation we have the elementary formula
    \begin{align}\label{eq-cos-cosh} 
  \labs{f(z)}^2=f(\Re z)^2 + \sinh^2(\Im z) \qquad \text{for $f\in\set{\sin,\cos}$}.
  \end{align}
  By the symmetry properties of $\sin$ and $\cos$ and the condition
  $\resconst\in(0,\frac{\pi}{2}]$ this gives both
  \begin{align*} 
    |f(z)|^2 \ %
    &\leq \ 1 + \sinh^2(\Im z) = \cosh^2(\Im z), \\ 
    |g(\lam\ang)|^2 \ %
    &\geq \ g(\ang\Re \lam)^2\geq \sin^2(\resconst)>0,
  \end{align*}
  which together proves \eqref{eq-f-cosh}.
  
  \medskip

  For the proof of \eqref{int-coscos} we consider two cases: We first assume
  that $\ang\labs{\Im\lam} \leq \rpconst$ holds. In particular
  $\labs{\lam\phi}\leq \ang|\lam|\leq 2\rpconst$ for $\phi\in (0,\ang)$ and
  hence $\labs{f(\lam\phi)}^2\leq \cosh^2(\phi\Im\lam)\leq \cosh^2(\rpconst)$ by
  \eqref{eq-f-cosh} and the symmetry and monotonicity of $\cosh$. Using
  \eqref{eq-f-cosh} we get
  \begin{align*}
    |\lam|\int_0^\ang \frac{|f(\lam \phi)|^2}{|g(\lam \ang)|^2} \ \d\phi \ %
    \leq \ \ang |\lam| \frac{\cosh^2(\rpconst)}{\sin^2(\resconst)} \ %
    \leq \ \frac{2\rpconst \cosh^2(\rpconst)}{\sin^2(\resconst)}.
  \end{align*}
  It remains to consider the case when
  $\ang\labs{\Im\lam} \geq \rpconst \geq \ang\labs{\Re\lam}$. Then
  $|\lam| \leq 2\labs{\Im\lam}$ and by \eqref{eq-cos-cosh} we get
  $|g(\lam\ang)|^2\geq \sinh^2(\ang\Im \lam)$.  Since
  $h(t):=(t + \sinh(t)\cosh(t))/\sinh^2(t)$ is monotonically decreasing for
  $t>0$, we arrive at
  \[ %
  |\lam|\int_0^\ang \frac{\labs{f(\lam \phi)}^2}{\labs{g(\lam \ang)}^2} \ \d\phi %
  \le
    \frac{|\lam|}{\sinh^2(\ang\Im\lam)}\int_0^\ang \cosh^2 (\phi\Im\lam)\ \d\phi =  \frac{|\lam| }{2\labs{\Im\lam}}h(\ang\labs{\Im\lam})\leq h(\rpconst).
  \]
\end{proof}
For the proof of the coercivity estimate we note the following simple fact:
  \begin{lemma}\label{lem-complane} %
    Let $\omega\in [0,\frac\pi 2]$. Then for all $z,w\in \C\setminus\set{0}$ with
    $|\arg z-\arg w|\le 2\omega$ there holds
    \begin{align*}
     \frac{|z+w|}{|z|+|w|}\ge \cos \omega.
    \end{align*}
  \end{lemma}
  \begin{proof}
    After rotation we may assume that $\Re z, \Re w\ge 0$ with $\tfrac{\Re z}{|z|}=\tfrac{\Re w}{|w|}=\cos\psi$ for some $\psi\in [0,\omega]$.
    Therefore we have
    \begin{align*}
    (|z|+|w|)\cos \omega &\le  (|z|+|w|) \cos \psi = \Re z + \Re w = \Re(z+w) \le |z+w|.
    \qedhere
    \end{align*}
  \end{proof}

\paragraph{Data Availability Statement.} Data sharing is not applicable to this article as no new data were created or analyzed in this study.
\paragraph{Acknowledgements.} MB acknowledges funding from the project \emph{Analysis of Moving Contact Lines (AnaCon)} (with project number OCENW.M20.194 of the research programme ENW - M) and MVG appreciates funding from the project \emph{Codimension two free boundary problems} (with project number VI.Vidi.223.019 of the research programme ENW - Vidi) both financed by the Dutch Research Council (NWO).
HK gratefully acknowledges support by Germany’s Excellence Strategy EXC-2181/1 – 390900948.
NM is supported by NSF grant DMS-1716466 and by Tamkeen under the NYU Abu Dhabi Research Institute grant of the center SITE.
FBR is supported by the Vici grant VI.C.212.027 of the NWO.

\small

\end{document}